\documentclass[letter, 10pt, oneside, roman]{amsart} %Sending notes, writing
\usepackage[utf8]{inputenc}
\usepackage{enumitem}
\usepackage{mathtools}
\usepackage{xcolor}
\usepackage[normalem]{ulem}

\usepackage{lipsum}

\usepackage{layout}

\usepackage{graphicx}
\usepackage{hyperref}
\usepackage{amssymb}
\usepackage{amsmath}
\usepackage{mathrsfs}
\usepackage[all]{xy}
\usepackage{latexsym}
\usepackage[colorinlistoftodos]{todonotes}
\usepackage[english]{babel}
\usepackage{csquotes}
\usepackage{tikz-cd} 
\usepackage{cancel}
\usepackage{enumitem}
\usepackage[most]{tcolorbox}
\usepackage{mdframed}
\usepackage{cleveref}

\usepackage{lmodern}

\usepackage{cancel}
\usepackage{enumitem}
\usepackage{marginnote}
\usepackage{appendix}
\allowdisplaybreaks

\newcommand{\Sp}{\textrm{Sp}}

\DeclarePairedDelimiterX{\inner}[2]{\langle}{\rangle}{#1, #2}

\DeclareMathOperator{\Aut}{Aut}

\DeclareMathOperator{\PGL}{PGL}

\DeclareFontFamily{U}{wncy}{}
\DeclareFontShape{U}{wncy}{m}{n}{<->wncyr10}{}
\DeclareSymbolFont{mcy}{U}{wncy}{m}{n}
\DeclareMathSymbol{\Sh}{\mathord}{mcy}{"58}

\newcommand{\Scal}{\mathcal{S}}

\newcommand{\SL}{{\rm SL}}
\newcommand{\GL}{{\rm GL}}

\newcommand{\SO}{{\rm SO}}

\newcommand{\Stab}{{\rm Stab}}

\renewcommand{\phi}{\varphi}

\newcommand{\Vol}{{\rm Vol}}
\newcommand{\Diff}{{\rm Diff}}

\newcommand{\Bcal}{ {\mathcal B}}
\newcommand{\Ocal}{ {\mathcal O}}

\newcommand{\Acal}{ {\mathcal A}}
\newcommand{\Ccal}{ {\mathcal C}}

\newcommand{\Sbb}{ {\mathbb S}}

\newcommand{\Cbb}{\mathbb{C}}

\newcommand{\Tbb}{{\mathbb T}}
\newcommand{\Rbb}{{\mathbb R}}
\newcommand{\Zbb}{{\mathbb Z}}
\newcommand{\Nbb}{{\mathbb N}}

\theoremstyle{cupplain}
\newtheorem{theorem}{Theorem}[section]
\newtheorem{lem}[theorem]{Lemma}
\newtheorem{coro}[theorem]{Corollary}
\newtheorem{nota}[theorem]{Notation}
\newtheorem{prop}[theorem]{Proposition}

\theoremstyle{cupdefinition}

\theoremstyle{cupremark}
\newtheorem{rmk}[theorem]{Remark}
\newtheorem{eg}[theorem]{Example}
\theoremstyle{cupproof}

\newcommand{\signat}{\textrm{sign}}
\newcommand{\Gr}{\textrm{Gr}}

\makeindex

\begin{document}

\title[Remarks on global rigidity]{Global rigidity of actions by higher rank lattices with dominated splitting.}
\author{Homin Lee}
\address{\noindent Indiana University Bloomington \newline 831 East 3rd St., Bloomington,\newline IN 47405}
%\curraddr{\href{mailto:hl63@indiana.edu}{hl63@indiana.edu}}
\email{\href{mailto:hl63@indiana.edu}{hl63@indiana.edu}}
\date{\today}
\maketitle\begin{abstract}

We prove that any smooth volume preserving action of a lattice $\Gamma$ in $\SL(n,\mathbb{R})$, $n\ge 3$, on a closed $n$-manifold which possess one element that admits a dominated splitting should be standard.  That means the manifold is the $n$-flat torus and the action is smoothly conjugate to an affine action. Note that an Anosov diffeomorphism, or more generally, a partial hyperbolic diffeomorphism admits a dominated splitting.  We obtained a topological global rigidity when $\alpha$ is $C^{1}$. We also prove similar theorems about an action of a lattice in $\Sp(2n,\Rbb)$ with $n\ge 2$ and $\SO(n,n)$, $n\ge 5$ on a closed $2n$-manifold.

\end{abstract}

\setcounter{tocdepth}{1}
\tableofcontents

\section{Introduction}

 In this paper, we will study topological and smooth global rigidity of higher rank lattice actions on specific dimensional manifold under a weak hyperbolicity assumption. Throughout the paper, a manifold is always a compact connected smooth Riemannian manifold without boundary. We can state one simple corollary of our result as following.
\begin{coro}\label{intro:coro}
Let $\Gamma<\SL(n,\Rbb)$ be a lattice for $n\ge 3$ and $\alpha\colon\Gamma\to \Diff^{1}(M^{n})$ be a volume preserving $C^{1}$ $\Gamma$ action on a closed smooth $n$ dimensional manifold $M$. Assume that 
there is a $\gamma_{0}\in \Gamma$ such that $\alpha(\gamma_{0})$ admits a dominated splitting.

Then the manifold $M$ is homeomorphic to a $n$-torus. Moreover, the action $\alpha$ is topologically conjugate to an affine action on the $n$-torus. If the action $\alpha$ is a $C^{\infty}$ action then the topological conjugacy is smooth.
\end{coro}

The conclusion says that, up to finite index, the action $\alpha$ conjugates with the automorphism action $\SL(n,\Zbb)$ on $\Tbb^{n}$ which has lots of Anosov diffeomorphisms.

Recall that, among others, an Anosov diffeomorphism or a partial hyperbolic diffeomorphism admits a dominated splitting. Even in the case that $\alpha(\gamma_{0})$ is an Anosov diffeomorphism, Corollary \ref{coro:main} gives a new rigidity result.

The existence of $C^{0}$ conjugacy implies that there is a finite index subgroup $\Gamma_{1}<\Gamma$ such that there is an unbounded group homomorphism from $\Gamma_{1}$ to $\SL(n,\Zbb)$. For instance, if $\Gamma$ is a cocompact, then such an action $\alpha$ can not exist. In other words, in that case, $\Gamma$ can not act on the $n$ manifold smoothly with an element which admits a dominated splitting.

The main motivations of this paper are twofold.

\begin{enumerate}
\item Recently Brown, Fisher and Hurtado proved that for a lattice $\Gamma$ in $\SL(n,\Rbb)$, $n\ge 3$, any volume-preserving smooth $\Gamma$ action on $n-1$ dimensional manifold should factor through finite group action. (e.g. \cite{BFH16} and \cite{BFH20}) The above corollary \ref{intro:coro} says that in an $n$-dimensional case, which is right after the critical dimension, if your action has a hyperbolic phenomenon, then the action should be standard action. 
\item Conjecturally, smooth lattice actions with one hyperbolic diffeomorphism conjugate to an algebraic action, see \cite{Goro} and \cite{Fis11}. Corollary \ref{intro:coro} gives an evidence for the conjecture.
\end{enumerate}

The global rigidity for higher rank lattice actions on manifold under hyperbolicity behavior studied many authors under various assumptions. We discuss some previous global rigidity results related to this paper.

First, without specifying base manifold, Feres proved that $C^{\infty}$ global rigidity for lattice in $\SL(n,\Rbb)$ ($n\ge 3$) action on $n$-manifold assuming that the action preserved connection, is non-isometric and is ergodic volume-preserving in \cite{FeresJDG}. In \cite{FeLab}, Feres and Labourie proved $C^{\infty}$ global rigidity result for a lattice in $\SL(n,\Rbb)$ ($n\ge 3$) action on $n$ manifold under certain assumptions including an Anosov property of the induced action on the suspension space. 
 
On the other hand, when the manifold is a torus (or more generally, a nilmanifold), Katok and Lewis proved that topological and smooth global rigidity of finite index subgroup of $\SL(n,\Zbb)$ action on the $n$-torus under certain assumptions including the existence of one Anosov element in \cite{KL96}. In \cite{KLZ}, Katok, Lewis, and Zimmer proved that smooth global rigidity of $\SL(n,\Zbb)$ action on $n$-torus assuming that the action has an ergodic fully supported probability measure and has one Anosov element. In \cite{GS99}, Goetze and Spatzier proved $C^{\infty}$ global rigidity of higher rank lattice volume-preserving Cartan actions on an arbitrary dimensional manifold. In \cite{MQ01}, Margulis and Qian proved topological global rigidity of higher rank lattice actions on nilmanifolds assuming that the action lifts to the universal cover, has an invariant measure, and has one Anosov element. Recently, Brown, Rodriguez Hertz, and Wang proved topological and smooth global rigidity of higher rank lattice actions on an arbitrary dimensional nilmanifolds only assuming a certain lifting condition and the existence of an Anosov element in \cite{BRHW17}.

\subsection{Statement of Main theorem}

In this paper, we will study about global rigidity of $\Gamma$ action on $M$ in the below cases \textbf{SL}, \textbf{Sp}, and \textbf{SO}. Roughly speaking, the main theorem says that in those cases, if we can see any uniform hyperbolicity from one element, then the entire group action should be algebraic. 
\begin{nota}[Case \textbf{SL}, \textbf{Sp}, and \textbf{SO}]\label{globnota}
We will focus on following cases.
\begin{enumerate}
\item[\textbf{SL}.] Let $n\ge 3$, $G=\SL(n,\Rbb)$, 
and $M$ be a closed $n$ dimensional manifold. Let $d=n$.
\item[\textbf{Sp}.] Let $n\ge 2$, $G=\Sp(2n,\Rbb)$, and $M$ be a closed $2n$ dimensional manifold. Let $d=2n$.

\item[\textbf{SO}.] Let $n\ge 5$, $G=\SO(n,n)$, and $M$ be a closed $2n$ dimensional manifold. Let $d=2n$.
\end{enumerate}
In the above cases, let $\Gamma<G$ be a lattice in $G$.
\end{nota}
Recall that the symplectic group $\Sp(2n,\Rbb)$ and the indefinite orthogonal group $\SO(n,n)$ is defined as follows. 
\[\Sp(2n,\Rbb)=\left\{g\in \SL(2n,\Rbb): g^{tr}Jg=J\right\},\quad J=\begin{bmatrix} 0_{n} & I_{n} \\ -I_{n} &0_{n} \end{bmatrix} \]
\[\SO(n,n)=\left\{g\in\SL(2n,\Rbb): g^{tr}I_{n,n}g=I_{n,n}\right\},\quad I_{n,n}=\begin{bmatrix} 0_{n}& I_{n} \\ I_{n} & 0_{n} \end{bmatrix}\] where $I_{n}$ is $n\times n$ identity matrix and $0_{n}$ is $n\times n$ zero matrix.

\begin{rmk}\label{rmk:sohigh} In the case of \textbf{SO}, we require that $n\ge 5$ although $n\ge 2$ is enough to get higher rank. When $n=2,3$, we have  isomorphisms \[\mathfrak{so}(2,2)\simeq \mathfrak{sl}(2,\Rbb)\times \mathfrak{sl}(2,\Rbb),\mathfrak{so}(3,3)\simeq \mathfrak{sl}(4,\Rbb).\] Therefore, in those cases, either the group is not simple or does not fit into the dimension condition. When $n=4$, $\textrm{Spin}(8,\Cbb)$ (or $\mathfrak{so}(8,\Cbb)$) admits $3$ nontrivial $8$ dimensional representations, defining and 2 half spin representations, so called \emph{Triality}. (e.g. \cite[Section 20.3]{FHrep}) We avoid this case for simplicity.

\end{rmk}
Recall that we say that $C^{1}$ diffeomorphism $f$ on $M$  \emph{admits a dominated splitting} if there is a $f$- invariant continuous nontrivial splitting $TM=E\oplus F$, and constants $C,\lambda>0$ such that for all  
$x\in M$ and $n\ge 0$, we have \[\frac{||D_{x}f^{n}(v)||}{||D_{x}f^{n}(w)||}<Ce^{-\lambda n}, \] for all $v\in E$, $w\in F$ with $||v||=||w||=1$. 
\begin{eg}
 Recall that $f\in \Diff^{1}(M)$ is called an \emph{Anosov diffeomorphism} if there is a $f$ invariant continuous nontrivial splitting $TM=E_{s}\oplus E_{u}$, and constants $C>0,\lambda<1$ such that for all $x\in M$ and $n\ge 0$, \[||D_{x}f^{n}(v)|| \le C\lambda^{n} ||v||, ||D_{x}f^{-n}(w)||\le C\lambda^{n}||w||,\] for all $v\in E_{s}$ and $w\in E_{u}$. All Anosov diffeomorphisms admit a dominated splitting. More generally partial hyperbolic diffeomorphism admits a dominated splitting.
 \end{eg}

Now the main result in this paper is the following.
\begin{theorem}[Main theorem]\label{thm:main}
Let $G$, $\Gamma$, $M$, and $d$ be as in Notation \ref{globnota}. Let $\alpha\colon\Gamma\to \Diff^{1}_{\Vol}(M)$ be a $C^{1}$ volume preserving action. 
Assume that there is $\gamma_{0}\in\Gamma$ such that $\alpha(\gamma_{0})$ admits dominated splitting.
Then the action has an Anosov element, i.e. \[\exists \gamma_{0}'\in\Gamma \textrm{ such that } \alpha(\gamma_{0}')\textrm{ is an Anosov diffeomorphism.}\]
Furthermore, in the case of \textbf{SL}, the manifold $M$ is homeomorphic to a torus. In the case of \textbf{Sp} and \textbf{SO}, the manifold $M$ is homeomorphic to a torus or an infra-torus. 
\end{theorem}
\begin{rmk} Indeed, in Theorem \ref{thm:main}, we prove that $\alpha(\gamma)$ is an Anosov diffeomorphism on $M$ for all hyperbolic element $\gamma\in\Gamma$. This will allow us to prove that the manifold is a torus or an infra-torus.
\end{rmk}

\subsection{Global rigidity of the action}
The main theorem implies that the action is indeed algebraic under a lifting assumption.
\begin{coro}\label{coro:main}
Let's retain same settings and notations as in Theorem \ref{thm:main}. In the case of \textbf{Sp} and \textbf{SO}, assume that there is a finite index subgroup $\Gamma_{0}$ in $\Gamma$ such that the action $\alpha|_{\Gamma_{0}}$ lifts to the finite cover $M_{0}$ of $M$ that is homeomorphic to torus. Denote the lifted action of $\Gamma_{1}$ on $M_{0}$ as $\alpha_{0}$. In the case of \textbf{SL}, simply denote $\Gamma_{0}=\Gamma$ and $M_{0}=M$.

Then, there is a finite subgroup $\Gamma_{1}<\Gamma_{0}$ such that the lifted action $\alpha_{0}$ of $\Gamma_{1}$ on $M_{0}$ is  topologically conjugate to its linear data $\rho_{0}$. More precisely, there is a finite index subgroup $\Gamma_{1}<\Gamma_{0}$ and a homeomorphism $h\colon M_{0}\to M_{0}$ such that \[\rho_{0}(\gamma)\circ h=\alpha_{0}(\gamma)\circ h,\quad \forall \gamma\in\Gamma_{1},\] where $\rho_{0}\colon\Gamma_{1}\to \Aut(M_{0})\simeq \SL(d,\Zbb)$ is the associated linear data of $\alpha_{0}$.
If $\alpha$ is a $C^{\infty}$ action then the conjugacy $h$ is indeed $C^{\infty}$ as well.

 \end{coro}
 
As we discussed earlier, Corollary \ref{coro:main} also says that such $\alpha$ can not exist unless the unbounded group homomorphism $\rho_{0}$ exists. Indeed, the existence of an unbounded homomorphism from finite index subgroup of $\Gamma$ to $\SL(d,\Zbb)$ is restrict. More precisely, up to finite index and after passing finite quotient, $\Gamma$ should be isomorphic to $G(\Zbb)$ due to Margulis' normal subgroup theorem and superrigidity theorem. Here $G(\Zbb)$ is $\SL(n,\Zbb)$, $\Sp(2n,\Zbb)$, or $\SO(n,n,\Zbb)$ in the case of \textbf{SL}, \textbf{Sp}, or \textbf{SO}, respectively.

In Corollary \ref{coro:main}, $h$ indeed conjugates the entire $\alpha_{0}(\Gamma_{0})$ action on $M_{0}$ with an affine action on $M_{0}$. (for example, \cite[Lemma 6.8.]{MQ01}.) Especially, in the case of \textbf{SL}, the entire $\Gamma$ action is conjugate to an affine action. Therefore, Theorem \ref{thm:main} and Corollary \ref{coro:main} imply Corollary \ref{intro:coro} easily.

In our setting, we do not assume the topological structure on the manifold except the dimension condition and the existence of an invariant measure. Furthermore, we only assumed the weak form of hyperbolicity, which is about the existence of dominated splitting. Indeed, if the dominated splitting is given by $TM=E\oplus F$ then we may not have any hyperbolicity on either $E$ or $F$. This is a much weaker assumption than being an Anosov or even a partial hyperbolic diffeomorphism.

\subsection{Organization of the paper}
The paper is organized as follows. In Section \ref{sec:pre} we provide some settings and preliminaries for the proof. In Section \ref{sec:aux}, we prove algebraic properties in the case of \textbf{SL}, \textbf{Sp}, and \textbf{SO}. We prove that we can arrange any two subspaces into general position using $G$ on one subspace. This will be used in the proof of Theorem \ref{thm:main}. In Section \ref{sec:main}, we prove Theorem \ref{thm:main}. The main idea is comparing continuous data from dominated splittings and measurable data from superrigidity. We extend and modify the idea in \cite{KLZ}. In Section \ref{sec:rig}, we prove Corollary \ref{coro:main} based on previous global rigidity results such as \cite{BRHW17}.

\subsection*{Acknowledgement} The author deeply appreciates Aaron Brown, David Fisher, and Ralf Spatzier for fruitful discussions regarding the problem. Especially, the author thanks to Aaron Brown for bringing this problem and the reference \cite{BMAnosov} to the author's attention and gave lots of comments on early draft. The author also thanks to Dylan Thurston for useful discussions about infra-torus.

\section{Preliminaries}\label{sec:pre}

We will always be in cases of \textbf{SL, Sp,} and \textbf{SO}. Throughout whole paper, we fix a vector space $V=\Rbb^{d}$ and the standard inner product on it where $d$ will be the number in Notation \ref{globnota}. In the case of \textbf{Sp} and \textbf{SO}, we will put additional structure on it.  
\subsection{Settings} 
We fix some notations which will be used throughout the paper.

We will denote the frame bundle $P$ over a $d$ dimensional manifold $M$. We identify the fiber of the frame bundle $P$ at $x\in M$ as the set of linear isomorphisms $V\to T_{x}M$.

In all cases, there are at most $3$ homomorphisms from $G$ to $\GL(V)$, up to conjugation, namely the trivial, defining, and contragredient representation due to dimension condition. Throughout the whole paper, we denote trivial, defining, and contragredient representation as $\pi_{0}$, $\pi_{1}$, and $\pi_{2}$ respectively.

In the case of \textbf{SL}, $\pi_{1}$ and $\pi_{2}$ are not isomorphic representation.  Nevertheless, we may assume that $\pi_{1}(G)=\pi_{2}(G)=\SL(V)$ as $\SL(n,\Rbb)$ is invariant under transpose. 

In the case of \textbf{Sp}, $\pi_{1}$ and $\pi_{2}$ are isomorphic as representation. We put the symplectic form $\omega$ and the standard symplectic basis $\Bcal$ on $V$. More precisely, we fix basis $\Bcal$ and a symplectic form $\omega$ on $V$ as  \[\Bcal=\{e_{1},\dots, e_{n}, f_{1},\dots,f_{n}\},\]   \[\omega(e_{i},f_{j})=\delta_{ij}, \quad \omega(e_{i},e_{j})=\omega(f_{i},f_{j})=0,\]  for all $i,j\in\{1,\dots, n\}$. Here $\delta_{ij}$ is the Kronecker delta symbol. Without loss of generality, we may assume that $\pi_{1}=\pi_{2}$ and \[\pi_{1}(G)=\pi_{2}(G)=\Sp(V,\omega)\] where \[\Sp(V,\omega)=\left\{A\in\SL(V): \omega(Av,Aw)=\omega(v,w)\textrm{ for all } v,w\in V.\right\}.\]

In the case of \textbf{SO},  $\pi_{1}$ and $\pi_{2}$ are isomorphic as representation again. We put a signature $(n,n)$ quadratic form $Q$ and a basis $\Ccal$ on $V$ as \[\Ccal=\{x_{1},\dots, x_{n},y_{1},\dots, y_{n}\},\] \[Q\left(\sum_{i=1}^{n}\left( a_{i}x_{i}+b_{i}y_{i}\right)\right)=\sum_{i=1}^{n}a_{i}^{2}-\sum_{i=1}^{n}b_{i}^{2}.\] Without loss of generality, we may assume that \[\pi_{1}(G)=\pi_{2}(G)=\SO(V,Q)\] where \[\SO(V,Q)=\left\{A\in\SL(V): Q(Av)=Q(v)\textrm{ for all } v\in V.\right\}.\]

Note that in all cases, $\pi_{1}$ and $\pi_{2}$ are irreducible.

\subsection{Cocycle Superrigidity Theorem}\label{sec:csr}

In the cases of \textbf{SL} and \textbf{Sp}, we have that $\SL_{n}$ and $\Sp_{2n}$ are algebraically simply connected. In the case of \textbf{SO}, the algebraically universal cover is $Spin_{2n}$. As we assume that $n\ge 5$, there is only one non-trivial $2n$ dimensional representation up to conjugation which is the defining representation of $\textrm{Spin}_{2n}(\Cbb)$. Furthermore, this representation factors through $\SO(2n,\Cbb)$. Therefore, based on discussions in the previous section, the special case of the non-ergodic version of cocycle superrigidity theorem \cite[Theorem 4.5]{FMW} can be stated as:

\begin{theorem}[Cocycle superrigidity, non-ergodic case \cite{FMW}]\label{thm:ZCSR}

Let $\Gamma$, $G$ and $M$ be as in the case of \textbf{SL}, \textbf{Sp}, or \textbf{SO}. Let $P$ be a frame bundle over $M$. Then there is a measurable section $\sigma\colon M\to P$,
measurable $\Gamma$-invariant map $\iota\colon M\to \{0,1,2\}$, a compact subgroups $\kappa_{i}\subset\GL(V)$ and measurable cocycles $K_{i}\colon\Gamma\times \iota^{-1}(i)\to \kappa_{i}$ for $i=0,1,2$ such that 
	\begin{enumerate}
		\item for $\mu$ almost every $x$ and every $\gamma\in\Gamma$, \[D_{x}(\gamma)\sigma(x)=
		\sigma(\alpha(\gamma)(x))\pi_{\iota(x)}(\gamma)K_{\iota(x)}(\gamma,x)
		\]  
		\item $\pi_{i}(G)$ commutes with $\kappa_{i}$ for $i=0,1,2$.
	\end{enumerate}
	\end{theorem}

Indeed, if we denote by $\{\mu_{i}\}_{i\in I}$ the ergodic decomposition of $\mu$ with respect to the $\Gamma$ action, then we have a partition of $I$ into $I_{0}$, $I_{1}$ and $I_{2}$ such that $\mu_{i}$ almost every $x$ we have $\iota(x)=i$ if and only if $i\in I_{i}$ for $i\in\{0,1,2\}$.  Note that $\iota^{-1}(i)$ is $\alpha(\Gamma)$ invariant measurable set for all $i\in\{0,1,2\}$. If we denote by \[\mu_{i}=\int_{j\in I_{i}} \mu_{j},\] for all $i\in \{0,1,2\}$ then we have decomposition of $\mu$ as \[\mu=\mu_{0}+\mu_{1}+\mu_{2}.\] Then the support of $\mu_{i}$, denoted by $X_{i}$, is a compact $\Gamma$-invariant subset in $M$ for all $i\in\{0,1,2\}$.

Furthermore,  as we know that any compact subgroup in $\GL(V)$ can be contained in $\textrm{O}(V)$ up to conjugation, we can assume that the compact subgroup $\kappa_{i}$ only depends on $\pi_{i}$.

In all cases, we have  \[Z_{\GL(d,\Rbb)}\left(\pi_{1}(G)\right)=Z_{\GL(d,\Rbb)}\left(\pi_{2}(G)\right)=\Rbb^{\times}\cdot I_{V}.\] Therefore, $\kappa_{1}$ and $\kappa_{2}$ are compact subgroups of $\Rbb^{\times}\cdot I_{V}$, so that they are subgroups of $\{\pm I_{V}\}$. Especially, $K_{1}$ and $K_{2}$ take values in $\{\pm I_{V}\}$.

In \cite[Theorem 4.5]{FMW}, the statement is not written in the bundle theoretic form. However, Theorem \ref{thm:ZCSR} can be directly deduced from it.

\subsection{Franks--Newhouse and Brin--Manning theorems}
In this section, we recall theorems about manifold admitting Anosov diffeomorphisms with certain properties. The first theorem due to Franks and Newhouse theorem states that if $M$ admits a codimension $1$ Anosov diffeomorphism, then it should be homeomorphic to torus (See \cite{Hir}, \cite{Newhouse} and \cite{FranksAnosov}):
\begin{theorem}[The Franks--Newhouse theorem]\label{thm:FNAnosov}
 If $f\colon M\to M$ is a codimension one Anosov diffeomorphism then $M$ is homeomorphic to a torus.
 
\end{theorem}

If we can control the dilation on stable and unstable distribution, then we can use the Brin--Manning theorem as follows. Let $M$ be a manifold and $f\colon M\to M$ be an Anosov diffeomorhpism. We can induce bounded linear operator $f_{*}$ on the Banach space of continuous vector field on $M$ as \[f_{*}X(x)=D_{f^{-1}(x)}(X(f^{-1}(x)))\] for continuous vector field $X$ on $M$. Then the spectrum of $f$ is contained in the interior of two annuli with the radius $0<\lambda_{1}<\lambda_{2}<1$ and $1<\mu_{2}<\mu_{1}<\infty$. 
\begin{theorem}[The Brin--Manning theorem, see \cite{BMAnosov}]\label{thm:BM}
Let $f\colon M\to M$ be an Anosov diffeomorphism on $M$. Let $\mu_{1},\mu_{2}$ and $\lambda_{1},\lambda_{2}$ be as above. Assume that \[1+\frac{\log\mu_{2}}{\log\mu_{1}}>\frac{\log\lambda_{1}}{\log\lambda_{2}},\] and \[1+\frac{\log\lambda_{2}}{\log\lambda_{1}}>\frac{\log\mu_{1}}{\log\mu_{2}}.\] Then $M$ is homeomorphic to a torus or a flat manifold.
\end{theorem}

Note that if $M$ is a flat manifold then $M$ is an infra-torus or torus by Biebebach's theorem \cite{flat1, flat2}.
\section{Auxiliarily lemmas for SL, Sp, and SO}\label{sec:aux}

The proof of the main theorem requires some properties of $\SL(V,\omega)$, $\Sp(V,\omega)$ and $\SO(V,Q)$. In this section, we will prove such facts. We retain notations in the previous section. 

Firstly, we will use the following fact that is about normalizer in $\GL(V)$.
\begin{lem}\label{lem:Ferlin}
Let $H\subset \GL(V)$ be an real algebraic subgroup in $\GL(V)$. Let $G$ be either $\SL(V)$, $\Sp(V,\omega)$, or $\SO(V,Q)$. Assume that $H$ is normalized by $G$. Then $H$ contains $G$ or is contained in $\Rbb^{\times}\cdot\{I_{V}\}$.
\end{lem}
Lemma \ref{lem:Ferlin} is a special case of the \cite[Lemma 2.5]{FeresJDG}. Indeed, one can prove by direct calculation as we are in explicit setting.

We will also need the following theorem which asserts the existence of nice subgroup $\Zbb^{n-1}$ in $\Gamma$.
In all cases, $G$ is a $\mathbb{R}$-split group. Denote the $\Rbb$-rank of $G$ as $\textrm{rk}_{\Rbb}(G)$. Note that \[\textrm{rk}_{\Rbb}(\SL(n,\Rbb))=n-1,\textrm{ and } \textrm{rk}_{\Rbb}(\Sp(2n,\Rbb))=\textrm{rk}_{\Rbb}(\SO(n,n))=n.\]  The following is a special case of the theorem by Prasad and Rapinchuk.
\begin{theorem}[Prasad--Rapinchuk, \cite{PR01}]\label{thm:PR}
 In all cases in Notation \ref{globnota}, there is a subgroup $\Gamma_{0}< \Gamma$ and a maximal $\Rbb$-split tori $A\simeq\Rbb^{\textrm{rk}_{\Rbb}(G)}$ in $G$ such that $\Acal=A\cap \Gamma_{0}\simeq \Zbb^{\textrm{rk}_{\Rbb}(G)}$ is a lattice in $A$. 
 \end{theorem}

Following sections, we prove that any two subspaces in $V$ can be made into general position by the action of $\Sp(V,\omega)$ or $\SO(V,Q)$ on one subspace. Recall that two subspace $W_{1}$ and $W_{2}$ are in general position if \[\dim(W_{1}\cap W_{2})=\max\{0,\dim W_{1}+\dim W_{2}-\dim V\}.\] Note that this property can be checked easily in the case of $\SL(V)$ as $\SL(V)$ acts transitively on every Grassmannian variety of $V$.

\subsection{Properties of \textbf{Sp}}  Note that $d=2n$ in this case. Recall that, in the case of \textbf{Sp}, we put the standard symplectic form $\omega$ and the standard symplectic basis \[\Bcal=\{e_{1},\dots, e_{n},f_{1},\dots, f_{n}\}\] on $V$. Also, $\Sp(V,\omega)$ is denoted by the symplectic group for the standard symplectic form $\omega$ on $V$. 

For any subspace $W<V$, We will denote \[W^{\perp}=\left\{v\in V: \omega(v,w)=0\textrm{ for all } w\in W.\right\}.\] Also, for any subspace $W<V$, the \emph{rank of $\omega$ on $W$} is defined as \[\textrm{rk}(\omega_{W})=\dim \left(W\cap W^{\perp}\right).\]  The following lemma can be checked directly by definition. 
\begin{lem}\label{lem:spsub} For any subspace $W<V$, the $\textrm{rk}(\omega_{W})$ is an even number. Moreover, for  $p,r\in \Nbb_{\ge 0}$, there is a subspace $W<V$ such that \[\dim W=p, \quad \textrm{rk}\left(\omega_{W}\right)=2r\] if and only if \[4r\le 2p\le 2r+2n.\]
\end{lem}

The next lemma says that if two subspaces have same dimension and rank, then we can find an element in $\Sp(V,\omega)$ so that it transforms one subspace to the other subspace.
\begin{lem}\label{lem:spact}
For any two subspaces $W_{1}, W_{2}<V$, if \[\dim W_{1}=\dim W_{2},\quad \textrm{rk}(\omega_{W_{1}})=\textrm{rk}(\omega_{W_{2}})\] then there is a $h\in \Sp(V,\omega)$ such that  \[hW_{1}=W_{2}.\]

\end{lem}

\begin{proof}
Assume that two subspaces $W_{1},W_{2}<V$ satisfy $\dim W_{1}=\dim W_{2}$ and $\textrm{rk}\left(\omega_{W_{1}}\right)=\textrm{rk}\left(\omega_{W_{2}}\right)$. We have Witt-Artin decomposition (see for example, \cite[Theorem 7.3.3]{DGp1}) that says if subspaces $H_{1}$ and $J_{1}$ that satisfy \[W_{1}= H_{1}\oplus\left(W_{1}\cap W_{1}^{\perp}\right),\quad W_{1}^{\perp}=J_{1}\oplus \left(W_{1}\cap W_{1}^{\perp}\right)\] then $H_{1}$ and $J_{1}$ are symplectic. Furthermore, we can $V$ can be decomposed into $\omega$-orthogonal direct sum \[V=H_{1}\oplus J_{1}\oplus (H_{1}\oplus J_{1})^{\perp}.\] Moreover, $W_{1}\cap W_{1}^{\perp}$ is a Lagrangian subspace in $(H_{1}\oplus J_{1})^{\perp}$. 

Similarly, for subspaces $H_{2}$ and $J_{2}$ that satisfy  \[W_{2}= H_{2}\oplus\left(W_{2}\cap W_{2}^{\perp}\right),\quad W_{2}^{\perp}=J_{2}\oplus \left(W_{2}\cap W_{2}^{\perp}\right)\] then $H_{2}$ and $J_{2}$ are symplectic. Furthermore, we can $V$ can be decomposed into $\omega$-orthogonal direct sum \[V=H_{2}\oplus J_{2}\oplus (H_{2}\oplus J_{2})^{\perp}.\] Moreover, $W_{2}\cap W_{2}^{\perp}$ is a Lagrangian subspace in $(H_{2}\oplus J_{2})^{\perp}$. 

Note that the above decomposition shows that we can decompose $\omega$ as \[\omega=\omega_{H_{1}}+\omega_{J_{1}}+\omega_{(H_{1}\oplus J_{1})^{\perp}} = \omega_{H_{2}}+\omega_{J_{2}}+\omega_{(H_{2}\oplus J_{2})^{\perp}}.\] Furthermore, by our assumptions about dimension and rank, \[\dim H_{1}=\dim H_{2}, \dim J_{1}=\dim J_{2}, \dim \left(W_{1}\cap W_{1}^{\perp}\right)= \dim \left(W_{2}\cap W_{2}^{\perp}\right).\]

Now we can find $h\in \Sp(V, \omega)$ such that \[h(H_{1})=H_{2}, h(J_{1})=J_{2},\textrm{ and } h\left(\left(H_{1}\oplus J_{1}\right)^{\perp}\right)=\left(H_{2}\oplus J_{2}\right)^{\perp}.\] Furthermore, as $W_{1}\cap W_{1}^{\perp}$ and $W_{2}\cap W_{2}^{\perp}$ are Lagrangians in $\left(H_{1}\oplus J_{1}\right)^{\perp}$ and $\left(H_{2}\oplus J_{2}\right)^{\perp}$ respectively, we can further require that $h$ satisfies \[h\left(W_{1}\cap W_{1}^{\perp}\right)=W_{2}\cap W_{2}^{\perp}.\] This implies that \[h(W_{1})=W_{2}.\] This proves the claim. \end{proof}

Next, we see that if there are two subspaces, then we can find subspaces that are in general position with same dimensions and ranks respectively.
\begin{lem}\label{lem:spgen}
Assume that $p,q\in\Nbb_{>0}$, $r,s \in \Nbb_{\ge 0}$ satisfies following conditions.
\begin{enumerate}
\item $4r\le 2p \le 2r+2n$ 
\item $4s\le 2q\le 2s+2n$
\item $p+q=2n$
\end{enumerate}
Then there are two subspaces $Y_{1},Y_{2}<V$ such that
\begin{enumerate}
\item $Y_{1}$ and $Y_{2}$ satisfy \[\textrm{rk}(\omega_{Y_{1}})=2r,\textrm{rk}(\omega_{Y_{2}})=2s, \dim Y_{1}=p, \dim Y_{2}=q,\] and
\item $Y_{1}$ and $Y_{2}$ are in general position.
\end{enumerate}
\end{lem}
\begin{proof}
We use the induction on $n$ for the above claim. For $n=2$, it is enough to check the following cases without loss of generality. 
\begin{enumerate}
\item When $p=1$, $2r=0$, $q=3$, and $2s=2$, we can take \[Y_{1}=\textrm{span}\{e_{1}\}, Y_{2}=\textrm{span}\{f_{1},e_{2},f_{2}\}.\]
\item When $p=2$, $2r=2$, $q=2$, and $2s=2$, we can take \[Y_{1}=\textrm{span}\{e_{1},f_{1}\}, Y_{2}=\textrm{span}\{e_{2},f_{2}\}.\]
\item When $p=2$, $2r=2$, $q=2$, and $2s=0$, we can take \[Y_{1}=\textrm{span}\{e_{1},f_{1}\}, Y_{2}=\textrm{span}\{e_{2}-e_{1},f_{2}+f_{1}\}.\]
\item When $p=2$, $2r=0$, $q=2$, and $2s=0$, we can take \[Y_{1}=\textrm{span}\{e_{1},e_{2}\}, Y_{2}=\textrm{span}\{f_{1},f_{2}\}.\]

\end{enumerate}
For notational convention, let $V^{n}$ be the $2n$ dimensional vector space with standard symplectic form and basis \[\{e_{1},\dots, e_{n},f_{1},\dots, f_{n}\}.\] Let $V^{n-1}$ be the symplectic subspace that is spanned by \[\{e_{1},\dots, e_{n-1},f_{1},\dots, f_{n-1}\}.\] Let $p,q,r,s$ satisfies the above assumptions for $n$, i.e. $p+q=2n$, $2r\le p\le r+n$, and $2s\le q\le s+n$. If $p>2r$ and $q>2s$ then we can find subspace $Y_{1}'$ and $Y_{2}'$ such that 
\[\textrm{rk}(\omega_{Y_{1}'})=2r,\textrm{rk}(\omega_{Y_{2}'})=2s, \dim Y_{1}'=p-1, \dim Y_{2}'=q-1,\]
and they are in general position as subspaces in $V^{n-1}$ by induction hypothesis. Then we can take $Y_{1}$ and $Y_{2}$ as \[Y_{1}=Y_{1}'\oplus \Rbb\{e_{n}\}, Y_{2}=Y_{2}'\oplus \Rbb\{f_{n}\}.\]
If $p=2r$ and $q>2s>0$ then we can find subspaces $Y_{1}'$ and $Y_{2}'$ such that 
\[\textrm{rk}(\omega_{Y_{1}'})=2r,\textrm{rk}(\omega_{Y_{2}'})=2s-2, \dim Y_{1}'=p, \dim Y_{2}'=q-2,\]
and they are in general position. Then we can take $Y_{1}$ and $Y_{2}$ as \[Y_{1}=Y_{1}', Y_{2}=Y_{2}'\oplus \Rbb\{e_{n},f_{n}\}.\] 
Assume that $p=2r$ and $q=2s$. Note that in this case $r+s=n$.  Then we can take \[Y_{1}=\textrm{span}\{e_{1},\dots,e_{r},f_{1},\dots, f_{r}\},Y_{2}=\textrm{span}\{e_{r+1},\dots, e_{r+s},f_{r+1},\dots, f_{r+s}\}.\] 
Finally, assume that $p=2r$, $q>0$ and $2s=0$. This implies that $q\le n$ so that $n\le p$, equivalently $n-r\le r$. Then we can take \[Y_{1}=\textrm{span}\{e_{1},\dots,e_{r},f_{1},\dots, f_{r}\},\]\[Y_{2}=\textrm{span}\{e_{r+1}-e_{1},\dots, e_{n}-e_{n-r},f_{r+1}+f_{1},\dots, f_{n}+f_{n-r}\}.\] 
The above case by case arguments and induction show that we proved the lemma. \end{proof}
Finally, we can deduce the following corollary from Lemma \ref{lem:spsub}, \ref{lem:spact} and \ref{lem:spgen}.
\begin{coro}\label{coro:spgoal}
For any two subspaces $W_{1},W_{2}<V$, we can find $h\in \Sp(V,\omega)$ such that $hW_{1}$ and $W_{2}$ are in general position. More precisely, we can find $h\in \Sp(V,\omega)$ such that 
\begin{enumerate} 
\item $hW_{1}\cap W_{2}=0$ if $\dim W_{1}+\dim W_{2}\le 2n$, 
\item $V=hW_{1}+W_{2}$ if $\dim W_{1}+\dim W_{2} \ge 2n$.
\end{enumerate}
\end{coro}
\begin{proof}
Denote $\dim W_{1}=s$ and $\dim W_{2}=k$.

If $k,s \ge n$ then we can find $n$-dimensional subspace $X_{1}<W_{1}$ and $X_{2}<W_{2}$. If $k,s\le n$ then we can find $n$-dimensional subspace $X_{1}>W_{1}$ and $X_{2}>W_{2}$. If $k\le n\le s$ and $k+s \ge 2n$ then we can find $n$-dimensional subspace $X_{1}<W_{1}$ and $X_{2}>W_{2}$. Finally, if $k\le n\le s$ and $k+s \le 2n$, then we can take $X_{1}=W_{2}$ and $2n-k$ dimensional subspace $X_{2}$ such that $X_{2} >W_{1}$. For the $s\le n\le k$ case, we can say similarly.

In all cases, using Lemma \ref{lem:spact} and \ref{lem:spgen}, one can find $h\in \Sp(V,\omega)$ such that $hX_{1}$ and $X_{2}$ are in general position. This implies that $hW_{1}$ and $W_{2}$ are in general position.
\end{proof}

\subsection{Properties of \textbf{SO}} 
Note that $d=2n$ in this case. Recall that, in the case of \textbf{SO}, we put a basis \[\Ccal=\{x_{1},\dots, x_{n},y_{1},\dots, y_{n}\}\] and a quadratic form $Q$ \[Q\left(\sum_{i=1}^{n}\left( a_{i}x_{i}+b_{i}y_{i}\right)\right)=\sum_{i=1}^{n}a_{i}^{2}-\sum_{i=1}^{n}b_{i}^{2}\] on $V$. Also, $\SO(V,Q)$ is denoted by the group of linear isomorphisms on $V$ that preserves $Q$. We call that an element in $\SO(V,Q)$ as \emph{isometry}.

For any subspace $W<V$, we denote the signature of the restricted quadratic form $Q|_{W}$ on $W$ as \[\signat(W)=(l,p,q)\] when $Q|_{W}$ admits orthogonal diagonalization \[Q|_{W}(x_{1},\dots, x_{l+p+q})= x_{l+1}^{2}+\dots +x_{l+p}^{2}-x_{l+p+1}^{2}-\dots -x_{l+p+q}^{2},\] for some basis in $W$. For instance, \[\signat(V)=(0,n,n).\] Note that the Sylvester's law of inertia says that the signature is the well-defined invariant for the quadratic form. (See \cite[Theorem 6.8]{Jac1}.)  

Recall that the quadratic form $Q$ gives bilinear symmetric form $B$ on $V$ as \[B(x,y)=Q(x+y)-Q(x)-Q(y),\] for all $x,y\in V$. For any subspace $W$ in $V$, we can define the perp of $W$ as \[W^{\perp}=\{v\in V: B(v,w)=0\textrm{ for all } w\in W.\}.\] Also, we denote \emph{orthogonal direct sum} of two subspaces $W_{1}$ and $W_{2}$ in $V$ as \[W_{1}\oplus W_{2}=W_{1}\perp W_{2}\] when $B(w_{1},w_{2})=0$ for all $w_{1}\in W_{1}$ and $w_{2}\in W_{2}$. 

The following lemma characterizes the possible signatures. 
\begin{lem}\label{lem:sosub} Let $W<V$ be a subspace of $V$. Let $\signat(W)=(l,p,q)$. Then \begin{enumerate}
\item $l+p\le n$, $l+q\le n$, and
\item $\dim W=l+p+q$.
\end{enumerate}
Conversely, assume that a triple of numbers $(l,p,q)$ satisfies $l,p,q \ge 0$, $l+p\le n$, $l+q \le n$, and $l+p+q \le 2n$. Then we can find a subspace $W$ in $V$ such that $\signat(W)=(l,p,q)$.
\end{lem}
\begin{proof} 
Let $W<V$ be a subspace with $\signat(W)=(l,p,q)$. If $l=0$ and $p >n$ then we can find $(n+1)$ linearly independent vectors $w_{1},\dots, w_{n+1}\in W$ such that $Q(w_{i})=1$. However, as $\signat(V)=(0,n,n)$, this gives contradiction. Same arguments hold for $q$. 

When $l \ge 1$, denote $\textrm{rad}(W)=W\cap W^{\perp}$ and fix a basis $\{w_{1},\dots, w_{r}\}$ of $\textrm{rad}(W)$. Write $W=W'\oplus \textrm{rad}(W)$ for some $W'$. Then we can find a subspace $U$ and a basis $z_{1},\dots, z_{r}$ of $U$ satisfies following conditions. (See \cite[Theorem 6.11]{Jac1}.)
\begin{enumerate}
\item the restricted quadratic form $Q|_{W\oplus U}$ is non-degenerate,
\item $\dim \textrm{rad}(W)=\dim U=r$,
\item For each $i\in\{1,\dots,r\}$, a pair $(z_{i},w_{i})$ spans hyperbolic plane $H_{i}$, and \[W\oplus U=W'\perp H_{1}\perp\dots \perp H_{r}.\]
\end{enumerate}
Recall that a two dimensional subspace is called hyperbolic plane if the restriction of $Q$ on it is non-degenerate and it contains a vector $u$ so that $Q(u)=0$.
This implies that \[\signat(W\oplus U)= (0,l+p,l+q).\] As $W\oplus U$ is a subspace of $V$, the arguments for $l=0$ case can be applied for $W\oplus U$ so that we can deduce $l+p, l+q \le n$.

Conversely, if we have triple numbers $(l,p,q)$ as in the statement, then we can just take subspace $W$ as \[W=\textrm{span}\{x_{1}+y_{1},\dots, x_{l}+y_{l},x_{l+1},\dots, x_{l+p},y_{l+1},\dots, y_{l+q}\}.\] \end{proof}

The next lemma says that if two subspaces have same signature, then we can find an element in $\SO(V,Q)$ so that it transforms one subspace to the other subspace.
\begin{lem}\label{lem:soact}
For any two subspaces $W_{1}, W_{2}<V$, if \[\signat(W_{1})=\signat(W_{2})\] then there is a $h\in \SO(V,\omega)$ such that  \[hW_{1}=W_{2}.\]

\end{lem}

\begin{proof}
Let $W_{1}$ and $W_{2}$ be subspaces in $V$ with same signature. As they have same signature, we can find a linear isometry $h$ between $W_{1}$ and $W_{2}$. Then the Witt's extension theorem says that $h$ can be extended to a linear isometry on $V$, i.e. an element in $\SO(V,Q)$. (See \cite[p.369]{Jac1}.) That means that we can find $h\in \SO(V,Q)$ such that $hW_{1}=W_{2}$ as we desired.
 \end{proof}

Next, we see that if there are two subspaces, then we can find subspaces that are in general position with same signatures respectively.
\begin{lem}\label{lem:sogen}
Assume that $l_{1},p_{1},q_{1}, l_{2}, p_{2},q_{2}\in\Nbb_{>0}$ satisfies following conditions.
\begin{enumerate}
\item $l_{1}+p_{1}, l_{1}+q_{1}\le n$,
\item $l_{2}+p_{2}, l_{2}+q_{2}\le n$, and
\item $l_{1}+p_{1}+q_{1}+l_{2}+p_{2}+q_{2}=2n$.
\end{enumerate}
Then there are two subspaces $Y_{1},Y_{2}<V$ such that
\begin{enumerate}
\item $Y_{1}$ and $Y_{2}$ satisfy \[\signat(Y_{1})=(l_{1},p_{1},q_{1}), \quad \signat(Y_{2})=(l_{2},p_{2},q_{2})\textrm{ and}\] 
\item $Y_{1}$ and $Y_{2}$ are in general position.
\end{enumerate}
\end{lem}
\begin{proof} 
Recall that we fix a basis \[\Ccal=\{x_{1},\dots, x_{n},y_{1},\dots, y_{n}\}\] on $V$ such that \[Q(x_{i})=1, Q(y_{i})=-1\] for all $i\in \{1,\dots, n\}$. We construct such $Y_{1}$ and $Y_{2}$ case by case. Without loss of generality, we only need to check both
\begin{enumerate}
\item[I.] $p_{1}\ge q_{1}$ and $p_{2}\le q_{2}$ case and
\item[II.] $p_{1}\le q_{1}$ and $p_{2}\le q_{2}$ case.
\end{enumerate}
Indeed, when $p_{1}\ge q_{1}$ and $p_{2}\ge q_{2}$, one can change roles of $x$ and $y$ from the case II.

\textbf{I. $p_{1}\ge q_{1}$ and $p_{2}\le q_{2}$ case:} Without loss of generality, we may assume \[p_{1}-q_{1}\le q_{2}-p_{2}.\] Indeed, we can change roles of $x$ and $y$ from below for the other case.

\textbf{I-1. When we are in the case of $p_{1}+p_{2}, q_{1}+q_{2} \le n$:} 
\begin{enumerate}
\item If $0\le l_{1}-n+q_{1}+q_{2}, 0\le l_{2}-n+q_{1}+q_{2}$, then we can take $Y_{1}$ and $Y_{2}$ as follows.
\begin{align*} 
Y_{1}&=\textrm{span}\left\{ x_{p_{1}+p_{2}+1}+y_{q_{1}+1},\dots, x_{n-l_{2}}+y_{2q_{1}+l_{1}-n+q_{2}}\right\}\\
&\perp \textrm{span}\left\{ x_{q_{1}+q_{2}+1}+y_{q_{1}+q_{2}+1},\dots, x_{n}+y_{n}\right\}\\
&\perp  \underbrace{\textrm{span}\left\{x_{1},\dots, x_{q_{1}},x_{q_{1}+p_{2}+1},\dots, x_{p_{2}+p_{1}}\right\}}_{\textrm{postive definite}}\\
&\perp \underbrace{\textrm{span}\left\{y_{1},\dots, y_{q_{1}}\right\}}_{\textrm{negative definite}}\\
Y_{2}&= \textrm{span}\left\{x_{n-l_{2}+1}+y_{1},\dots, x_{q_{1}+q_{2}}+y_{q_{1}+q_{2}+l_{2}-n}\right\}\\
&\perp \textrm{span}\left\{ x_{q_{1}+q_{2}+1}-y_{q_{1}+q_{2}+1},\dots, x_{n}-y_{n}\right\}\\
&\perp \underbrace{\textrm{span}\left\{y_{q_{1}+1},\dots,y_{p_{2}+q_{2}}\right\}}_{\textrm{negative definite}}\\
&\perp \underbrace{\textrm{span}\left\{x_{q_{1}+1},\dots, x_{q_{1}+p_{2}}\right\}}_{\textrm{positive definite}}
\end{align*}
\item If $l_{1}-n+q_{1}+q_{2}\le 0\le l_{2}-n+q_{1}+q_{2}$, then we can take $Y_{1}$ and $Y_{2}$ as follows.
\begin{align*}
Y_{1}&=\textrm{span}\left\{x_{n-(l_{1}-1)}+y_{n-(l_{1}-1)},\dots, x_{n}+y_{n}\right\}\\
&\perp\underbrace{\textrm{span}\left\{x_{q_{1}+p_{2}+1},\dots, x_{p_{1}+p_{2}},x_{1},\dots, x_{q_{1}}\right\}}_{\textrm{positive definite}}\\
&\perp\underbrace{\textrm{span}\left\{y_{1},\dots, y_{q_{1}}\right\}}_{\textrm{negative definite}}\\
Y_{2}&=\textrm{span}\left\{ x_{p_{1}+p_{2}+1}+y_{1},\dots, x_{n-l_{1}}+y_{n-(l_{1}+p_{1}+p_{2})}\right\}\\
&\perp \textrm{span}\left\{x_{1}+y_{q_{1}+q_{2}+1},\dots, x_{n-(l_{1}+q_{1}+q_{2})}+y_{n-l_{1}}\right\}\\
&\perp\textrm{span}\left\{x_{n-(l_{1}-1)}-y_{n-(l_{1}-1)},\dots, x_{n}-y_{n}\right\}\\
&\perp\underbrace{\textrm{span}\left\{x_{q_{1}+1},\dots, x_{q_{1}+p_{2}}\right\}}_{\textrm{postive definite}}\\
&\perp\underbrace{\textrm{span}\left\{y_{q_{1}+1},\dots, y_{q_{1}+p_{2}}\right\}}_{\textrm{negative definite}}
\end{align*}
\item If $l_{2}-n+q_{1}+q_{2}\le 0\le l_{1}-n+q_{1}+q_{2}$, then we can take $Y_{1}$ and $Y_{2}$ as follows.
\begin{align*} 
Y_{1}&= \textrm{span}\left\{x_{n-(l_{2}-1)}+y_{n-(l_{2}-1)},\dots, x_{n}+y_{n}\right\}\\
&\perp \textrm{span}\left\{x_{p_{1}+p_{2}+1}+y_{q_{1}+1},\dots, x_{n-l_{2}}+y_{n-l_{2}-p_{2}-p_{1}}\right\}\\
&\perp\textrm{span}\left\{x_{q_{1}+1}+y_{q_{1}+q_{2}+1},\dots, x_{n-l_{2}-q_{2}}+y_{n-l_{2}}\right\}\\
&\perp \underbrace{\textrm{span}\left\{x_{q_{1}+p_{2}+1},\dots, x_{p_{1}+p_{2}},x_{1},\dots, x_{q_{1}}\right\}}_{\textrm{positive definite}}\\
&\perp \underbrace{\textrm{span}\left\{y_{1},\dots, y_{q_{1}}\right\}}_{\textrm{negative definite}}\\
Y_{2}&=\textrm{span}\left\{x_{n-(l_{2}-1)}-y_{n-(l_{2}-1)},\dots, x_{n}-y_{n}\right\}\\
&\perp \underbrace{\textrm{span}\left\{x_{q_{1}+1},\dots, x_{q_{1}+p_{2}}\right\}}_{\textrm{positive definite}}\\
&\perp \underbrace{\textrm{span}\left\{y_{q_{1}+1},\dots, y_{q_{1}+q_{2}}\right\}}_{\textrm{negative definite}}
\end{align*}
\end{enumerate}

\textbf{I-2. When we are in the case of $q_{1}+q_{2}\le n\le p_{1}+p_{2}$:}  We can take $Y_{1}$ and $Y_{2}$ as follows.
\begin{align*}
Y_{1}&=\textrm{span}\left\{ x_{p_{1}+p_{2}+q_{1}+q_{2}+l_{2}-n+1}+y_{q_{1}+1},\dots, x_{n}+y_{q_{1}+l_{1}}\right\}\\
&\perp\underbrace{\textrm{span}\left\{x_{1},\dots, x_{q_{1}},x_{q_{1}+p_{2}+1},\dots, x_{p_{1}+q_{2}}\right\}}_{\textrm{positive definite}}\\
&\perp \underbrace{\textrm{span} \left\{y_{1},\dots, y_{q}\right\}}_{\textrm{negative definite}}\\
Y_{2}&=\textrm{span}\left\{x_{p_{1}+p_{2}+q_{1}+q_{2}-n+1}+y_{q_{1}+q_{2}-n+1},\dots, x_{p_{1}+p_{2}+q_{1}+q_{2}-n+l_{2}}+y_{q_{1}+q_{2}+l_{2}-n}\right\} \\
&\perp \underbrace{\textrm{span}\left\{x_{q_{1}+1},\dots, x_{q_{1}+p_{2}}\right\}}_{\textrm{positive definite}}\\
&\perp\underbrace{\textrm{span}\left\{x_{p_{1}+q_{2}+1}+\sqrt{2}y_{1},\dots, x_{p_{1}+p_{2}+q_{1}+q_{2}-n}+\sqrt{2}y_{q_{1}+q_{2}-n}\right\}}_{\textrm{negative definite}}\\
&\perp \underbrace{\textrm{span}\left\{y_{q_{1}+1},\dots, y_{n}\right\}}_{\textrm{negative definite}}
\end{align*}

\textbf{I-3. When we are in the case of $q_{1}+q_{2}\ge n\ge p_{1}+p_{2}$:} One can construct desired $Y_{1}$ and $Y_{2}$ from the case I-2. after changing roles of $x$ and $y$.

\textbf{II. $p_{1}\le q_{1}$, $p_{2}\le q_{2}$ case:} 
Without loss of generality, we may assume that $l_{1}\le l_{2}$. We divide into two cases.

\textbf{II-1. When we are in the case of $p_{1}+p_{2}\le n$:}
\begin{enumerate}
\item If $n-p_{1}-p_{2}\le l_{1}$ and $n-p_{1}-p_{2}\le l_{2}$, then we can take $Y_{1}$ and $Y_{2}$ as follows.
\begin{align*}
Y_{1}&=\textrm{span}\left\{ x_{p_{1}+p_{2}+1}+y_{p_{1}+p_{2}+1},\dots, x_{n}+y_{n}\right\}\\
&\perp \textrm{span}\left\{x_{p_{1}+1}+y_{q_{1}+q_{2}+1},\dots, x_{p_{1}+l_{1}-n+p_{1}+p_{2}}+y_{q_{1}+q_{2}+l_{1}-n+p_{1}+p_{2}}\right\}\\
&\perp\underbrace{\textrm{span}\left\{x_{1},\dots, x_{p_{1}}\right\}}_{\textrm{positive definite}}\\
&\perp \underbrace{\textrm{span} \left\{y_{1},\dots, y_{q_{1}}\right\}}_{\textrm{negative definite}}\\
Y_{2}&=\textrm{span}\left\{x_{p_{1}+p_{2}+1}-y_{p_{1}+p_{2}+1},\dots, x_{n}-y_{n}\right\}\\
&\perp\textrm{span}\left\{x_{1}+y_{q_{1}+q_{2}+l_{1}-n+p_{1}+p_{2}+1},\dots, x_{l_{2}-n+p_{1}+p_{2}}+y_{p_{1}+p_{2}}\right\}\\
&\perp\underbrace{\textrm{span}\left\{x_{p_{1}+1},\dots, x_{p_{1}+p_{2}}\right\}}_{\textrm{positive definite}}\\
&\perp\underbrace{\textrm{span}\left\{y_{q_{1}+1},\dots, y_{q_{1}+q_{2}}\right\}}_{\textrm{negative definite}}
\end{align*}

\item If $l_{1}\le n-p_{1}-p_{2}\le l_{2}$, then we can take $Y_{1}$ and $Y_{2}$ as follows.
\begin{align*}
Y_{1}&=\textrm{span}\left\{x_{n-(l_{1}-1)}+y_{n-(l_{1}-1)},\dots, x_{n}+y_{n}\right\}\\
&\perp \underbrace{\textrm{span}\left\{x_{1},\dots, x_{p_{1}}\right\}}_{\textrm{positive definite}}\\
&\perp \underbrace{\textrm{span}\left\{y_{1},\dots, y_{q_{1}}\right\}}_{\textrm{negative definite}}\\
Y_{2}&=\textrm{span}\left\{x_{n-(l_{1}-1)}-y_{n-(l_{1}-1)},\dots, x_{n}-y_{n}\right\}\\
&\perp\textrm{span}\left\{x_{p_{1}+p_{2}+1}+y_{1},\dots, x_{n-l_{1}}+y_{n-l_{1}-p_{1}-p_{2}}\right\}\\
&\perp\textrm{span}\left\{x_{1}+y_{q_{1}+q_{2}+1},\dots, x_{n-l_{1}-q_{1}-q_{2}}+y_{n-l_{1}}\right\}\\
&\perp\underbrace{\textrm{span}\left\{x_{p_{1}+1},\dots, x_{p_{1}+p_{2}}\right\}}_{\textrm{positive definite}}\\
&\perp\underbrace{\textrm{span}\left\{y_{q_{1}+1},\dots, y_{q_{1}+q_{2}}\right\}}_{\textrm{negative definite}}
\end{align*}
\end{enumerate}

\textbf{II-2. When we are in the case of $p_{1}+p_{2}\ge n$:} In this case we can take $Y_{1}$ and $Y_{2}$ as follows.
\begin{align*}
Y_{1}&=\textrm{span}\left\{x_{p_{1}+1}+y_{p_{1}+p_{2}+q_{1}+q_{2}-n+1},\dots, x_{l_{1}+p_{1}}+y_{l_{1}+p_{1}+p_{2}+q_{1}+q_{2}-n}\right\}\\
&\perp\underbrace{\textrm{span}\left\{x_{1},\dots,x_{p_{1}}\right\}}_{\textrm{positive definite}}\\
&\perp\underbrace{\textrm{span}\left\{y_{1},\dots, y_{q_{1}}\right\}}_{\textrm{negative definite}}\\
Y_{2}&=\textrm{span}\left\{x_{1}+y_{l_{1}+p_{1}+p_{2}+q_{1}+q_{2}-n+1},\dots, x_{l_{2}}+y_{n}\right\}\\
&\perp\underbrace{\textrm{span}\left\{x_{p_{1}+1},\dots,x_{n}\right\}}_{\textrm{positive definite}}\\
&\perp\underbrace{\textrm{span}\left\{\sqrt{2}x_{1}+y_{q_{1}+q_{2}+1},\dots, \sqrt{2}x_{p_{1}+p_{2}-n}+y_{p_{1}+p_{2}+q_{1}+q_{2}-n}\right\}}_{\textrm{positive definite}}\\
&\perp\underbrace{\textrm{span}\left\{y_{q_{1}+1},\dots, y_{q_{1}+q_{2}}\right\}}_{\textrm{negative definite}}
\end{align*}

The above case by case constructions show that we can find desired subspaces $Y_{1}$ and $Y_{2}$. \end{proof}

As a result, we can deduce the following corollary that is same analogue with Corollary \ref{coro:spgoal}. The proof is same except that we use Lemma \ref{lem:soact} and \ref{lem:sogen} instead of \ref{lem:spact} and \ref{lem:spgen}.
\begin{coro}\label{coro:sogoal}
For any two subspaces $W_{1},W_{2}<V$, we can find $h\in \SO(V,Q)$ such that $hW_{1}$ and $W_{2}$ are in general position. More precisely, we can find $h\in \SO(V,Q)$ such that 
\begin{enumerate} 
\item $hW_{1}\cap W_{2}=0$ if $\dim W_{1}+\dim W_{2}\le 2n$, 
\item $V=hW_{1}+W_{2}$ if $\dim W_{1}+\dim W_{2} \ge 2n$.
\end{enumerate}
\end{coro}

\section{Proof of the Main theorem}\label{sec:main}
In this section, we prove Theorem \ref{thm:main}. We will always denote $d$, $\Gamma$, $G$, and $M$ as in Notation \ref{globnota}. We further retain notations and settings in previous sections. As before, we fix a vector space $V=\Rbb^{d}$ and identify the fiber of frame bundle $P$ at $x\in M$ as a group of linear isomorphisms $V\to T_{x}M$.  

Let $\alpha\colon\Gamma\to \Diff^{1}(M)$ be a $C^{1}$ action. We assumed that following holds. 
	\begin{enumerate}
		\item There is a $\gamma_{0}\in\Gamma$ so that $\alpha(\gamma_{0})$ admits dominated splitting. We can find continuous $\alpha(\gamma_{0})$ invariant subbundle $E$ and $F$, some constant $C,\lambda>0$ so that \[TM=E\oplus F, \quad \frac{||D_{x}\alpha(\gamma_{0}^{n})(v)||}{||D_{x}\alpha(\gamma_{0}^{n})(w)||}<Ce^{-\lambda n}, \quad \forall n>0\] for any unit vectors $v\in E$ and $w\in F$.
		\item There is a fully supported $\alpha(\Gamma)$ invariant Borel probability measure $\mu$. 
	\end{enumerate}

As in Section \ref{sec:csr}, let $\pi_{0}$, $\pi_{1}$ and $\pi_{2}$ be the trivial, defining and contragredient representations of $G$ on $V$, respectively. Recall that if $\pi_{1}$ is isomorphic to $\pi_{2}$ then we assumed $\pi_{1}=\pi_{2}$ after conjugation. Further, in all cases, we assume that $\pi_{1}(G)=\pi_{2}(G)$ after conjugation. Then the cocycle superrigidity theorem \ref{thm:ZCSR} says that we can find a measurable section $\sigma\colon M\to P$,
measurable map $\iota\colon M\to \{0,1,2\}$, a compact subgroup $\kappa_{i}\subset\GL(V)$ and measurable cocycle $K_{i}\colon \Gamma\times \iota^{-1}(i)\to \kappa_{i}$ for $i=0,1,2$ such that 
	\begin{enumerate}
		\item for $\mu$ almost every $x\in M$ and for every $\gamma\in\Gamma$, \[D_{x}(\gamma)\sigma(x)=
		\sigma(\alpha(\gamma)(x))\pi_{\iota(x)}(\gamma)K_{\iota(x)}(\gamma,x)
		\]  with respect to the measurable framing $\sigma$.
		\item $\pi_{i}(G)$ commutes with $\kappa_{i}$ for $i=0,1,2$.
	\end{enumerate}

Recall that $K_{1}$ and $K_{2}$ takes values on $\{\pm I_{V}\}$.

As the proof involves, we organize steps in the proof here. We prove that $\iota^{-1}(0)$ is $\mu$-measure zero set in Section \ref{sec:nul} using the dominated splitting. This allows us to focus on nontrivial representations $\pi_{1}$ and $\pi_{2}$.

In Section \ref{sec:contconj}, we prove that, after projectivization, the measurable section is same as continuous section almost everywhere using a comparison between measurable data from the superrigidity theorem and continuous data from the dominated splitting. Until this point, everything works for fully supported $\Gamma$- invariant Borel probability measure $\mu$. And then, under the volume preserving assumption, we conclude that for any hyperbolic element $\gamma\in\Gamma$,  $\alpha(\gamma)$ is an Anosov diffeomorphism.

After that, in Section \ref{sec:M}, using the Franks--Newhouse theorem in the case of \textbf{SL}, and the Brin--Manning theorem in the case of \textbf{Sp} and \textbf{SO}, we will deduce that $M$ is a torus or infra-torus. This will give a proof of Theorem \ref{thm:main}.

\subsection{Nullity of $\iota^{-1}(0)$}\label{sec:nul}
We prove a following lemma that says the set $\iota^{-1}(0)$ is a null set under our assumptions so that we can ignore the appearance of $\pi_{0}$ and $\kappa_{0}$.

\begin{lem}\label{lem:scrapply}
Under same notations and assumptions in Theorem \ref{thm:main}, we can find a measurable section $\sigma\colon M\to P$ and a measurable map $\iota\colon M\to \{1,2\}$ such that for $\mu$ almost every $x$, every $\gamma\in\Gamma$,
 \[D_{x}(\alpha(\gamma))\sigma(x)=\pm \sigma(\alpha(\gamma)(x))\pi_{\iota(x)}(\gamma),\] for $\mu$ almost every $x\in M$ and every $\gamma\in \Gamma$.  \end{lem}

\begin{proof}
Recall that we fixed standard inner product on $V=\Rbb^{d}$.   For the sake of contradiction, assume that the $\iota^{-1}(0)$ has positive $\mu$ measure.

As $\pi_{0}$ is the trivial representation, we have for any $n\ge 0$ and $\mu$-almost every $x\in \iota^{-1}(0)$, \begin{align}\label{eq2} D_{x}(\alpha(\gamma_{0}^{n}))\sigma(x)=\sigma(\alpha(\gamma_{0}^{n})(x))K_{0}(\gamma_{0}^{n},x),\end{align} where $K_{0}: \Gamma\times \iota^{-1}(0)\to \kappa_{0}$ is the compact group valued cocycle.  

Fix Riemannian metric on $M$. For any $c>0$, we define measurable subset $\Scal_{c}\subset M$ as \[\Scal_{c}=\left\{x\in \iota^{-}(0):\textrm{ $x$ satisfies (\ref{eq2}) and }  ||\sigma(x)||, ||\sigma(x)^{-1}|| <c\right\}.\] Here $||\sigma(x)||$ is the operator norm  of $\sigma(x):V\to T_{x}M$ with respect to standard norm on $V$ and Riemannian metric on $T_{x}M$. Since we assumed that $\mu(\iota^{-1}(0))$ has positive measure, we can find sufficiently large number $C_{0}>0$ such that \[\mu\left(\Scal_{C_{0}}\right)>0.\]  

Using the Poincar\'e recurrence theorem for measure preserving transformation $\alpha(\gamma_{0})$, for $\mu$-almost every  $x_{0}\in S_{C_{0}}$, one can find sequence $\left\{n_{k}\right\}_{k\in\Nbb}\subset \Nbb$ such that 
\begin{enumerate}
\item $n_{0}=0$, $n_{k}\to \infty$ as $k\to \infty$, 
 \item equation (\ref{eq2}) holds, and
 \item for all $k\ge0$, \[||\sigma(\alpha(\gamma_{0}^{n_{k}})(x_{0}))||<C_{0}, \textrm{ and } ||\sigma(\alpha(\gamma_{0}^{n_{k}})(x_{0}))^{-1}||<C_{0}.\] 
 \end{enumerate}
Fix $x_{0}\in S_{C_{0}}$ and sequence $\left\{n_{k}\right\}$ as above.

It follows that \[\lim_{k\to\infty}\frac{1}{n_{k}}\ln||D_{x_{0}}(\alpha(\gamma_{0}^{n_{k}}))(v)||=\lim_{k\to\infty}\frac{1}{n_{k}}\ln||D_{x_{0}}(\alpha(\gamma_{0}^{n_{k}})(w)|| =0,\] for any unit vector $v\in E_{x_{0}}$ and $w\in F_{x_{0}}$. Indeed, for any unit vector $v\in E$, denote $v_{V}=\sigma(x)^{-1}(v)\in V$. Then using (\ref{eq2}) and the bound of $\sigma$, we can get \[\frac{1}{C}||v_{V}||<||D_{x_{0}}(\alpha(\gamma_{0}^{n_{k}})(v)||=||\sigma(\alpha(\gamma_{0}^{n_{k}})(x_{0}))K_{0}(\gamma_{0}^{n_{k}},x_{0})(v_{V})||<C||v_{V}||.\] This implies \[\lim_{k\to\infty}\frac{1}{n_{k}}\ln||D_{x_{0}}(\alpha(\gamma_{0}^{n_{k}}))(v)||=0.\] Similar argument can be applied for the unit vector $w\in F$. This proves claim.

However, as $\gamma_{0}\in \Gamma$ admits a dominated splitting, for all $x\in M$, we have \[\left[\lim_{k\to\infty}\frac{1}{n_{k}}\ln||D_{x}(\alpha(\gamma_{0}^{n_{k}}))(v)||\right] -\left[\lim_{k\to\infty}\frac{1}{n_{k}}\ln||D_{x}(\alpha(\gamma_{0}^{n_{k}})(w)|| \right]<-\lambda.\] 

This gives contradiction.  \end{proof}

\subsection{Continuity of the measurable conjugacy}\label{sec:contconj}  In this section, we modify and extend the main ideas in \cite{KLZ}.

We use the same notations as in the previous section. Recall that we fix $V=\Rbb^{d}$ and identify the fiber of the frame bundle $P$ over $M$ at $x$ as the set of linear isomorphisms $V\to T_{x}M$. Lets denote projective frame bundle as $\overline{P}\to M$. Each fibers can be naturally identified with $\PGL(V)$. We will use the bracket notation $[-]$ for the projectivization of a linear map.

We prove the following continuity statement in this section. Recall that an element $g\in G$ is called \emph{hyperbolic} if there is no eigenvalue (over $\Cbb$) of $g$ with modulus $1$ and \emph{semisimple} if $g$ is diagonalizable over $\Rbb$.

Recall that, in Section \ref{sec:csr}, only in the case of \textbf{SL}, $\pi_{1}$ is not isomorphic to $\pi_{2}$. In this case, we saw that the measure $\mu$ can be decomposed into \[\mu=\mu_{1}+\mu_{2}\] where \[\mu_{1}(\iota^{-1}(1))=1, \mu_{2}(\iota^{-1}(2))=1\] unless $\mu=\mu_{1}$ or $\mu=\mu_{2}$. We denoted $\alpha(\Gamma)$ invariant compact set $X_{1}$ and $X_{2}$ as the support of $\mu_{1}$ and $\mu_{2}$, respectively. As $\mu$ is fully supported, $M=X_{1}\cup X_{2}$. When $\mu=\mu_{1}$ or $\mu=\mu_{2}$, $X_{2}=\emptyset$ or $X_{1}=\emptyset$, respectively. For simplicity, when $\pi_{1}=\pi_{2}$, set $X_{1}=M$ and $X_{2}=\emptyset$.

\begin{prop}[Continuity of measurable conjugacy]\label{prop:scrcont}
Let $G$, $\Gamma$, $M$, and $d$ be as in Notation \ref{globnota}. Let $\alpha\colon\Gamma\to \Diff^{1}(M)$ be a $C^{1}$ action and $\mu$ be a $\alpha(\Gamma)$ invariant fully supported Borel probability measure. Then, 
there are continuous (global) sections to projective frame bundle over $M$, $\Ccal_{i}\colon X_{i}\to \overline{P}$ for $i\in\{1,2\}$ such that \[[D_{x}\alpha(\gamma)]\Ccal_{i}(x)=\Ccal_{i}(\alpha(\gamma)(x))[\pi_{i}(\gamma)]\] for any $x\in X_{i}$ and $\gamma\in \Gamma$. 
 \end{prop}

Enumerate elements of $\Gamma$ into $\Gamma=\{\gamma_{0},\gamma_{1},\dots\}$ where $\gamma_{0}$ is the distinguished element that $\alpha(\gamma_{0})$ admits dominated splitting. Then for any $\gamma_{j}\in \Gamma$, $\alpha(\gamma_{j}\gamma_{0}\gamma_{j}^{-1})$ admits dominated splitting. More precisely, there is $C_{j},\lambda_{j}>0$ such that \[TM=E^{j}\oplus F^{j},\quad \frac{||D_{x}(\alpha(\gamma_{j}\gamma_{0}^{n}\gamma_{j}^{-1}))(v)||}{||D_{x}(\alpha(\gamma_{j}\gamma_{0}^{n}\gamma_{j}^{-1}))(w)||}< C_{j}e^{-\lambda_{j}n},\] for all $x\in M$, all unit vectors $ v\in E^{j}, w\in F^{j}$, and $n\ge 0$. Note that \[E_{j}=\alpha(\gamma_{j})E,\textrm{ and } F_{j}=\alpha(\gamma_{j})F.\] Therefore, $\dim E^{j}$ and $\dim F^{j}$ does not depend on $j$ so that we can denote $\dim E^{j}=s, \dim F^{j}=u$.

Denote the projectivization of the measurable section $\sigma$ in Lemma \ref{lem:scrapply} as $C\colon M\to \overline{P}$. Note that $C$ is just measurable section. Lemma \ref{lem:scrapply} says that the derivative is (measurably) conjugate to representation $\pi_{1}$ or $\pi_{2}$ up to $\pm I_{V}$. Therefore, we can find non trivial subspaces $W_{E,1}^{j}, W_{E,2}^{j}, W_{F,1}^{j}$ and  $W_{F,2}^{j}$ in $V$ for $\mu$-almost every $x\in M$ such that \[C(x)W_{E,1}^{j}=E_{x}^{j}, \quad C(x)W_{F,1}^{j}=F_{x}^{j} \textrm{ for a.e. } x\in \iota^{-1}(1)\] and \[C(x)W_{E,2}^{j}=E_{x}^{j}, \quad C(x)W_{F,2}^{j}=F_{x}^{j} \textrm{ for a.e. } x\in \iota^{-1}(2).\] 

Furthermore, $W_{E,*}^{j}$ and $W_{F,*}^{j}$ are the corresponding spaces for the dominated splitting of the linear map $\pi_{*}(\gamma_{j}\gamma_{0}\gamma_{j}^{-1})$ for $*\in\{1,2\}$ on $V$. So that \[W_{E,*}^{j}=\pi_{*}(\gamma_{j})W_{E,*}^{0},\textrm{ and } W_{F,*}^{j}=\pi_{*}(\gamma_{j})W_{F,*}^{0},\] for any $j$ and $*$. Note that for all $j$, \[\dim W_{E,1}^{j}=\dim W_{E,2}^{j}=s, \quad \dim W_{F,1}^{j}=\dim W_{F,2}^{j}=u.\]

We have the standard algebraic $\GL(V)$ action on each  Grassmannian varieties $\Gr(V,s)$ and $\Gr(V,u)$ consists of dimension $s$ and $u$ subspaces respectively. We denote the stabilizer of the action as $\Stab$ for simplicity. 

We prove the following lemma first.

\begin{lem}\label{lem:lin}
We can find $m\ge 0$ such that \[\bigcap_{j=0}^{m-1}\Stab(W_{E,*}^{j}),\bigcap_{j=0}^{m-1}\Stab(W_{F,*}^{j})\subset\Rbb^{\times}\cdot\{I_{V}\},\] for each $*\in\{1,2\}$.  Especially, their images in $\PGL(V)$ are all trivial. 
\end{lem}

\begin{proof}[Proof of Lemma \ref{lem:lin}]
 Define for each $*\in\{1,2\}$, 
\[S_{E,*}=\bigcap_{j=0}^{\infty} \Stab(W_{E,*}^{j})=\bigcap_{j=0}^{\infty}\Stab\left(\pi_{*}(\gamma_{j})W_{E,*}^{0}\right)=\bigcap_{j=0}^{\infty}\pi_{*}(\gamma_{j})\Stab(W_{E,*}^{0})\pi_{*}(\gamma_{j})^{-1}.\]
We claim that $S_{E,*}$ is contained in $\Rbb^{\times}\cdot \{I_{V}\}$.

Using Borel density theorem (\cite{Borel3}) asserts that 
\[S_{E,*}=\bigcap_{g \in G}\pi_{*}(g)\Stab(W_{E,*}^{0})\pi_{*}(g)^{-1}.\] 
This implies that $S_{E,*}$ is normalized by $\pi_{*}(G)$. Note that $S_{E,*}$ is an algebraic group defined over $\Rbb$. 

Recall that for the case \textbf{SL}, we assumed that \[\pi_{1}(G)=\pi_{2}(G)=\SL(V).\] For the case \textbf{Sp}, we put the standard symplectic form $\omega$ on $V$ so that \[\pi_{1}(G)=\pi_{2}(G)=\Sp(V,\omega).\] Finally, for the case \textbf{SO}, we put a quadratic form $Q$ with signature $(0,n,n)$ on $V$ so that \[\pi_{1}(G)=\pi_{2}(G)=\SO(V,Q).\]

Lemma \ref{lem:Ferlin} says that $S_{E,*}$ contains $\pi_{*}(G)$ or is contained in $\Rbb^{\times}\cdot \{I_{V}\}$. The irreducibility of $\pi_{*}$ says that $S_{E,*}$ should be contained in $\Rbb^{\times}\cdot \{I_{V}\}$. We can define and argue similarly for $S_{F,*}$ so that we have \[S_{F,*}=\bigcap_{j=0}^{\infty} \Stab(W_{F,*}^{j})\subset\Rbb^{\times}\cdot\{I_{V}\}.\]

Note that $\Stab(W_{E,*}^{j})$, $\Stab(W_{F,*}^{j})$, $S_{E,*}$, and $S_{F,*}$ are all real algebraic groups. By Noetherian property, in all cases, we can find $m$ such that \[S_{E,*}=\bigcap_{j=0}^{m-1}\Stab(W_{E,*}^{j}), \quad S_{F,*}=\bigcap_{j=0}^{m-1}\Stab(W_{F,*}^{j}),\] for each $*\in\{1,2\}$. Also, we already proved that \[S_{E,*},S_{F,*}\subset\Rbb^{\times}\cdot\{I_{V}\}.\] Especially, their images in $\PGL(V)$ are all trivial. \end{proof}

Lemma \ref{lem:lin} implies that if we consider $\PGL(V)$ action on $\Gr(V,s)$ and $\Gr(V,u)$ then the common stabilizer of $\{W_{E,*}^{j}\}_{j}$ is trivial as well as for common stabilizer of $\{W_{F,*}^{j}\}_{j}$.

From now on, we fix $m$ as in Lemma \ref{lem:lin}. Let $\Gr(s)$ and $\Gr(u)$ be the Grassmannian bundle over $M$ that consists dimension $s$ and $u$ subspaces of the tangent space, respectively. 
Then we can define the fiberwise action of $\PGL(T_{x}M)$ on $\Gr(T_{x}M,s)^{m}\times \Gr(T_{x}M,u)^{m}$. We denote the fiberwise orbit map for each $*\in\{1,2\}$ as
\[\Phi^{0}_{*}\colon \overline{P} \to  \left(\left(\Gr(s)^{m}\right)\times \left(\Gr(u)^{m}\right)\right)\]  \[\Phi^{0}_{*}\left(x,[p_{x}]\right)\]\[= \left(x,\left([p_{x}](W_{E,*}^{0}),\dots, [p_{x}](W_{E,*}^{m-1})\right), \left([p_{x}](W_{F,*}^{0}), \dots, [p_{x}](W_{F,*}^{m-1})\right)\right).\]
Lemma \ref{lem:lin} implies that the common stabilizer is trivial so that we can prove the followings.

\begin{lem}[Orbit map is local embedding and smooth.]\label{lem:injorb}
$\Phi^{0}_{*}$ is smooth injective local embedding map. Furthermore, the inverse map $\left(\Phi^{0}_{*}\right)^{-1}$ defined on $\Phi^{0}_{*}(\overline{P})$ is continuous.
\end{lem}
\begin{proof}[Proof of Lemma \ref{lem:injorb}] 
The regularity of $\Phi^{0}_{*}$ comes from the fact that the action is algebraic.   In addition, note that the algebraic action the orbit is locally closed with respect to Zariski topology so that with respect to Hausdorff topology. Now Lemma \ref{lem:lin} and the inverse function theorem shows that the map $\Phi^{0}_{*}$ is injective local embedding and the inverse map is continuous on $\Phi^{0}_{*}(\overline{P})$.  \end{proof}

Now we are ready to prove the proposition \ref{prop:scrcont}.
 \begin{proof}[Proof of the proposition \ref{prop:scrcont}]
Recall that we denoted $m$ as in Lemma \ref{lem:lin}.

Define the continuous map $\tau$ as
\[\tau\colon M\to \left(\left(\Gr(s)^{m}\right)\times \left(\Gr(u)^{m}\right)\right)\]  \[\tau(x)=\left(x,\left(E^{0}_{x},\dots, E^{m-1}_{x}\right),\left(F^{0}_{x},\dots, F^{m-1}_{x}\right)\right).\] 
Indeed, as the splitting is continuous, we know that the map $\tau$ is continuous.

Lets denote $U_{*}$ as the image of $\Phi^{0}_{*}$ for each $*\in\{1,2\}$ and $U=\bigcup_{*\in\{1,2\}}U_{*}$. Recall that we denote the measurable section $C:M\to \overline{P}$ as the projection of that $\sigma$ that comes from Lemma \ref{lem:scrapply}.  Then for $\mu$-almost every point $x\in \tau^{-1}(U)$, we have \[\tau(x)=\Phi^{0}_{*}(C(x)),\] where $\iota(x)=*$. As $\mu$ is fully supported, we can find $M^{0}\subset \tau^{-1}(U)$ so that 
\begin{enumerate}
\item $M^{0}$ is dense in $M$,
\item $\mu(M^{0})=1$, and
\item for all $x\in M^{0}$, we have  \[\tau(x)=\Phi^{0}_{*}(C(x)),           \] where $\iota(x)=*\in\{1,2\}$.
\end{enumerate}

First, we show that $\tau(M)\subset U$. Assume that $M\neq \tau^{-1}(U)$. Then for any $x_{0}\in M\setminus \tau^{-1}(U)$, we can find the sequence $\{x_{n}\}_{n\in\Nbb}\subset M^{0}$ so that $x_{n}\to x_{0}$ as $n\to \infty$.
We may find subsequence $\{x_{n_{k}}\}_{k}$ so that $\iota(x_{n_{k}})=1$ or $2$ for all $k$. For notational convention, just denote such a subsequence as $\{x_{n}\}$ so that $\iota(x_{n})=*$ for some fixed $*\in\{1,2\}$ for all $n$. 

Now as $x_{n}\in M^{0}$ for any $n$, we have \[\tau(x_{n})=\Phi^{0}_{*}(C(x_{n})).\] We use local trivialization at $x_{0}$. We can find a homeomorphism $\phi\colon  \Ocal\times \GL(V) \to P|_{\Ocal}$, so that $\overline{\phi}\colon \Ocal\times \PGL(V)\to\overline{P}|_{\Ocal}$ on an open neighborhood $\Ocal$ of $x_{0}$. For sufficiently large $n$, we may assume that $x_{n}\in\Ocal$. Then we can find sequence $g_{n}\in \GL(V)$ such that \begin{enumerate}
\item $\{g_{n}\}_{n\in\Nbb}$ is bounded, and
\item $\overline{\phi}(x_{n},[g_{n}])=C(x_{n})$.
\end{enumerate}

Passing to subsequence, we can find endomorphism on $V$, or equivalently $V$ matrix $L\in M_{V}(\Rbb)$ so that $g_{n}\to L$ as $n\to \infty$. If $L\in \GL(V)$, then using continuity of $\tau$ and $\Phi^{0}_{*}$,
\begin{align*}
 \tau(x_{0})&=\lim_{n\to\infty}\tau(x_{n})\\
 &=\lim_{n\to\infty}\Phi^{0}_{*}(C(x_{n}))=\lim_{n\to\infty}\Phi^{0}_{*}(\overline{\phi}(x_{n},[g_{n}]))\\
 &=\Phi^{0}_{*}(\overline{\phi}(x_{0},[L])).\end{align*}
  This contradicts the assumption that $x_{0}\in M\setminus \tau^{-1}(U)$. Therefore, $L\notin\GL(V)$ so that we can find nontrivial subspaces \[K=\ker L, R=L(V) <V.\]

We claim the following lemma that says we can make the subspaces are in general positions. We will provide the proof later.
\begin{lem}\label{lem:general}
There is a  $\gamma_{l}\in\Gamma$ so that \[\pi_{*}(\gamma_{l})W_{E,*}^{0}=W_{E,*}^{l} \textrm{ and } \pi_{*}(\gamma_{l})W_{F,*}^{0}=W_{F,*}^{l}\] are in general position with $K$.  
\end{lem}Also, as a Grassmannian variety is compact, passing to subsequence, we may assume that \[g_{n}W_{E,*}^{l} \to Q_{E}, \quad g_{n}W_{F,*}^{l}\to Q_{F}\quad \textrm{as } n\to \infty,\] for some $Q_{E}\in \Gr(V,s)$ and $Q_{F}\in\Gr(V,u)$.

If $W_{E,*}^{l}$ intersects with $K$ trivially, then \[\lim_{n\to\infty} g_{n}W_{E,*}^{l} = LW_{E,*}^{l}\subset LV=R.\] Otherwise, $W_{E,*}^{l} + K =V$ so that \[R=LV\subset Q_{E}=\lim_{n\to\infty} g_{n}W_{E,*}^{l}.\]

This implies either $Q_{E},Q_{F}\subset R$, $Q_{E}\subset R\subset Q_{F}$, $Q_{F}\subset R\subset Q_{E}$ or $R\subset Q_{E},Q_{F}$. In any cases, $Q_{E}$ and $Q_{F}$ are not transversal.

Now denote continuous map $\tau_{E}\colon M\to \Gr(s)$ and $\tau_{F}\colon M\to \Gr(u)$ as \[\tau_{E}(x)=E_{x}^{l}, \tau_{F}(x)=F_{x}^{l},\] using $\sigma_{0}$.

Then, by construction, we have \[\tau_{E}(x_{n})=\phi_{x_{n}}^{-1}(g_{n}W_{E,*}^{l}),  \tau_{F}(x_{n})=\phi_{x_{n}}^{-1}(g_{n}W_{F,*}^{l}).\]

 Using continuity of $\tau_{E}$ and $\tau_{F}$, we can deduce that \[ \tau_{E}(x_{0})=\phi_{x_{0}}^{-1}(Q_{E}), \tau_{F}(x_{0})=\phi_{x_{0}}^{-1}(Q_{F}).\] Therefore, $Q_{E}$ and $Q_{F}$ should be transversal. It gives contradiction. In summary, we prove that $M=\tau^{-1}(U)$. Indeed, the above argument shows that if $x\in M$ is the limit of a sequence $\{x_{n}\}_{n=1}^{\infty}\subset \iota^{-1}(i)\cap M^{0}$ then $\tau(x)\in U_{i}$ for all $i\in \{1,2\}$.

Recall that, $\mu$ can be decomposed into $\mu=\mu_{1}+\mu_{2}$ where $\mu_{1}(\iota^{-1}(1))=1$, $\mu_{2}(\iota^{-1}(2))=1$ unless $\mu=\mu_{1}$ or $\mu=\mu_{2}$. We denote $\alpha(\Gamma)$ invariant compact sets $X_{1}$ and $X_{2}$ as the support of $\mu_{1}$ and $\mu_{2}$, respectively. As $\mu$ is fully supported, $M=X_{1}\cup X_{2}$. 

Define the map $\Ccal_{1}:X_{1}\to \overline{P}$ and $\Ccal_{2}:X_{2}\to \overline{P}$ as for all $x\in X_{1}$ or $y\in X_{2}$, \[\Ccal_{1}(x)=\left(\Phi_{1}^{0}\right)^{-1}\circ \tau (x), \quad \Ccal_{2}(y)=\left(\Phi_{2}^{0}\right)^{-1}\circ \tau(y),\] respectively. Indeed the above argument shows not only $M=\tau^{-1}(U)$ but also $\Ccal_{1}$ and $\Ccal_{2}$ are well defined on $X_{1}$ and $X_{2}$ respectively. Furthermore, $\Ccal_{1}$ and $\Ccal_{2}$ are continuous by Lemma \ref{lem:injorb}. We have \[C(x)=\Ccal_{1}(x)=\left(\Phi_{1}^{0}\right)^{-1}\circ\tau(x),\] \[C(y)=\Ccal_{2}(y)=\left(\Phi_{2}^{0}\right)^{-1}\circ\tau(y)\] for all $x\in\iota^{-1}(1)\cap M^{0}$ and $y\in \iota^{-1}(2)\cap M^{0}$. 
This implies that for each $i\in\{1,2\}$, \[[D_{x}\alpha(\gamma)]\Ccal_{i}(x)=\Ccal_{i}(\alpha(\gamma)(x))[\pi_{i}(\gamma)]\] for all $\gamma\in \Gamma$ and for all $x\in X_{i}$.   This proves the proposition. 				\end{proof}

As we said earlier, we prove Lemma \ref{lem:general} here.
\begin{proof}[Proof of Lemma \ref{lem:general}] 
Note that \[U_{E,*}=\{g\in \pi_{*}(G): gW_{E,*}^{0} \textrm{ is in general position with $K$.}\}, \] and \[U_{F,*}=\{g\in \pi_{*}(G): gW_{F,*}^{0} \textrm{ is in general position with $K$.}\}\] are Zariski open subsets for any $*\in\{1,2\}$. 

In the case \textbf{SL}, as $\SL(V)$ acts transitively on $\Gr(s)$ and $\Gr(u)$ respectively, $U_{E,*}$ and $U_{F,*}$ are nonempty Zariski open subsets. Therefore, the intersection $U_{E,*}\cap U_{F,*}$ is also a nonempty Zariski open subset as $\SL(V)$ is an irreducible variety. Using Borel density theorem, we can find $\gamma_{l}\in\Gamma$ so that $\pi_{*}(\gamma_{l})W_{E,*}^{0}=W_{E,*}^{l}$ and $\pi_{*}(\gamma_{l})W_{F,*}^{0}=W_{F,*}^{l}$ are in general position with $K$.

In the case of \textbf{Sp} and \textbf{SO}, using Corollary \ref{coro:spgoal} and \ref{coro:sogoal}, respectively, we have \[U_{E,*}=\{g\in \pi_{*}(G): gW_{E,*}^{0} \textrm{ is in general position with $K$.}\}\neq\emptyset. \]

Therefore, $U_{E,*}^{0}$ is non-empty Zariski open in $\pi_{*}(G)$. Similarly, \[U_{F,*}=\{g\in \pi_{*}(G): gW_{F,*}^{0} \textrm{ is in general position with $K$.}\}\] is a non-empty Zariski open subset in $\pi_{*}(G)$. Therefore, using irreducibility of $\pi_{*}(G)$, $U_{E,*}^{0}\cap U_{F,*}^{0}$ is also a non-empty Zariski open so that Borel density theorem asserts that there is $\gamma_{l}\in \Gamma$ such that $\pi_{*}(\gamma_{l})W_{E,*}^{0}=W_{E,*}^{l}$ and $\pi_{*}(\gamma_{l})W_{F,*}^{0}=W_{F,*}^{l}$ are in general position with $K$. \end{proof}

Now we will find lots of Anosov diffeomorphims under the volume preserving assumption.

Recall that, using Theorem \ref{thm:PR}, we can find a finite index subgroup $\Gamma_{0}$ in $\Gamma$ and a maximal $\Rbb$-split torus $A$ in $G$ such that $\Acal=A\cap \Gamma_{0}$ is a cocompact lattice in $A$.

\begin{prop}[Abundance of Anosov diffeomorphisms]\label{prop:anos}
Let $G$, $\Gamma$, $M$, and $d$ be as in Notation \ref{globnota}. Let $\alpha\colon\Gamma\to \Diff^{1}_{\Vol}(M)$ be a $C^{1}$ volume preserving action. 
Assume that there is $\gamma_{0}\in\Gamma$ such that $\alpha(\gamma_{0})$ admits dominated splitting.
Then every hyperbolic element $\gamma\in\Gamma$, especially $\gamma\in\Acal$, is an Anosov diffeomorphism. \end{prop}
\begin{proof} Using Proposition \ref{prop:scrcont} for the measure with positive density,  we can find continuous (global) sections to projective frame bundle over $M$, $\Ccal_{i}\colon X_{i}\to \overline{P}$ for $i\in\{1,2\}$ such that \[[D_{x}\alpha(\gamma)]\Ccal_{i}(x)=\Ccal_{i}(\alpha(\gamma)(x))[\pi_{i}(\gamma)]\] for any $x\in X_{i}$ and $\gamma\in \Gamma$.
As we assumed $\alpha(\Gamma)$ preserves volume, for any $\gamma\in\Gamma$ and for each $i\in\{1,2\}$, we can find continuous map $\widetilde{\Ccal_{i}}:X_{i}\to P|_{X_{i}}$ such that for all $x\in X_{i}$, 

\[D_{x}\alpha(\gamma)\widetilde{\Ccal_{i}}(x)=\pm \widetilde{\Ccal_{i}}(\alpha(\gamma)(x))\pi_{i}(\gamma).\] 
Indeed, the Jacobian determinant of $D_{x}\alpha(\gamma)$ is $\pm 1$ and the determinant of $\pi_{i}(\gamma)$ is always $1$. 

For any hyperbolic element $\gamma\in\Gamma$, denote $E_{\pi_{1}(\gamma)}^{s}$ and $E_{\pi_{1}(\gamma)}^{u}$ by sum of generalized eigenspaces corresponding eigenvalue less than $1$ and bigger than $1$ of $\pi_{1}(\gamma)$ respectively. $E_{\pi_{2}(\gamma)}^{u}$ and $E_{\pi_{2}(\gamma)}^{u}$ can be defined similarly. We have decomposition \[V=E_{\pi_{1}(\gamma)}^{s}\oplus E_{\pi_{1}(\gamma)}^{u}=E_{\pi_{2}(\gamma)}^{s}\oplus E_{\pi_{2}(\gamma)}^{u}\] with respect to $\pi_{1}(\gamma)$ and $\pi_{2}(\gamma)$ respectively. Then we have splitting \[T_{x}M|_{X_{1}}=\Ccal_{1}E_{\pi_{1}(\gamma)}^{s}\oplus \Ccal_{1}E_{\pi_{1}(\gamma)}^{u},\] \[T_{x}M|_{X_{2}}=\Ccal_{2}E_{\pi_{2}(\gamma)}^{s}\oplus \Ccal_{2}E_{\pi_{2}(\gamma)}^{u}\] which is continuous on $X_{1}$ and $X_{2}$ respectively and constants $C'>0$, $0<\lambda <1$ such that
\[||D_{x}\alpha(\gamma^{n})v_{i}^{s}||<C'\lambda^{n}, ||D_{x}\alpha(\gamma^{-n})v_{i}^{u}||<C'\lambda^{n}\] for all $x\in X_{i}$, $v_{i}^{s}\in \Ccal_{i}E_{\pi_{i}(\gamma)}^{s}$, $v_{i}^{u}\in \Ccal_{i}E_{\pi_{i}(\gamma)}^{u}$, and all $n>0$.  This implies that $\alpha(\gamma)$ is an Anosov diffeomorphism which the unstable (stable) distribution comes from the eigenspaces with a modulus bigger than $1$  (less than $1$ resp.) eigenvalue of $\pi_{*}(\gamma)$. This finishes the proof of the proposition.  \end{proof}

\begin{rmk}\label{rmk:SLanos} Indeed, we can find a continuous map $\widetilde{\Ccal}:M\to P$

for each $\gamma\in \Gamma$ such that for all $x\in M$, \[D_{x}\alpha(\gamma)\widetilde{\Ccal}(x)=\pm \widetilde{\Ccal}(\alpha(\gamma)(x))\pi_{i}(\gamma),\] for some fixe $i\in\{1,2\}$.

Indeed, in the cases of \textbf{Sp} and \textbf{SO}, as $\pi_{1}=\pi_{2}$, this is obviously true.  In the case of \textrm{SL}, we prove the above claim even though $\pi_{1}$ is not isomorphic to $\pi_{2}$ as follows. We can find $\gamma\in\Acal$ such that the eigenvalues of $\pi_{1}(\gamma)$ are \[\lambda_{1}>1>\lambda_{2}>\dots>\lambda_{n}\] so that the eigenvalues of $\pi_{2}(\gamma)$ is \[\lambda_{n}^{-1}>\dots>\lambda_{2}^{-1}>1>\lambda_{1}^{-1}.\] (For more detailed explanations, see Section \ref{sec:case1}.) Let $E_{1}^{u}$ be the eigenspace corresponding to eigenvalue $\lambda_{1}$ of $\pi_{1}(\gamma)$. Further, let $E_{1}^{s}$ be the sum of eigenspaces corresponding to the eigenvalues $\lambda_{2},\dots, \lambda_{n}$ of $\pi_{1}(\gamma)$.  Similarly, let $E_{2}^{s}$ be the eigenspace corresponding to the eigenvalue $-\lambda_{1}$ of $\pi_{2}(\gamma)$ and $E_{2}^{u}$ be the sum of eigenspaces corresponding to eigenvalues $-\lambda_{n},\dots,-\lambda_{2}$ of $\pi_{2}(\gamma)$. Then we have decomposition \[V=E_{1}^{s}\oplus E_{1}^{u}=E_{2}^{s}\oplus E_{2}^{u}\] with respect to $\pi_{1}(\gamma)$ and $\pi_{2}(\gamma)$ respectively. 

The proof of Proposition \ref{prop:anos} says that $\alpha(\gamma)$ is an Anosov diffeomorphism with the dimension of stable manifold is $1$ in $X_{1}$ and $n-1$ in $X_{2}$. This gives contradiction with connectivity of the manifold $M$. Indeed, the stable distribution of an Anosov diffeomorphism varies continuously. Therefore, we proved that either $X_{1}$ or $X_{2}$ is empty. This proves the claim.

\end{rmk}

\subsection{Characterizing the manifold $M$}\label{sec:M}
We characterize $M$ in this subsection. We prove that $M$ is a torus or a flat manifold in the case of \textbf{SL}, \textbf{Sp} and \textbf{SO} under the assumptions in the theorem. 
 \subsubsection{Case \textbf{SL}}\label{sec:case1}
Theorem \ref{thm:PR} implies that we can find a semisimple hyperbolic element $\gamma\in \Acal$ such that every eigenvalue is positive and only one eigenvalue is bigger than $1$. Indeed, we can find the elements $a_{\lambda}\in A$ such that $a_{\lambda}$ is similar to $\textrm{diag}(\lambda^{n-1},\lambda^{-1},\dots,\lambda^{-1})$. As $\Acal=A\cap \Gamma_{0}\simeq\Zbb^{n-1}$ is a lattice in $A\simeq \Rbb^{n-1}$, we can find $\gamma\in A\cap \Gamma_{0}$ and $\lambda$ such that $\gamma$ and $a_{\lambda}$ are arbitrary close (as elements in $A$) as much as we want. Here, one can use the fact that \[\left\{v\in \Rbb^{n-1}: ||v||=1, tv\in \Zbb^{n-1}\textrm{ for some } t\in \Rbb.\right\} \] is dense in the unit sphere $\Sbb^{n-2}$ in $\Rbb^{n-1}$.  

If $\lambda$ is sufficiently large and $\gamma\in A\cap \Gamma_{0}$ is sufficiently close to $a_{\lambda}$ then $\gamma$ is a desired semisimple hyperbolic element.

 Proposition \ref{prop:anos} and Remark \ref{rmk:SLanos} imply that the $\alpha(\gamma)$ is a codimension one Anosov diffeomorphism. Therefore, the manifold $M$ is homeomorphic to the torus due to the Franks--Newhouse theorem \ref{thm:FNAnosov}. This proves Theorem \ref{thm:main} in the case of \textbf{SL}.

\subsubsection{Case \textbf{Sp} and \textbf{SO}}  Theorem \ref{thm:PR} implies that we can find a semisimple hyperbolic element $\gamma\in \Gamma_{0}$ such that the eigenvalues can be written as $e^{\lambda_{1}}>e^{\lambda_{2}}>\dots >e^{\lambda_{n}}>1>e^{-\lambda_{n}}>\dots>e^{-\lambda_{1}}$ then \[1+\frac{\lambda_{n}}{\lambda_{1}}>\frac{\lambda_{1}}{\lambda_{n}},\quad1+\frac{\lambda_{n}}{\lambda_{1}}>\frac{\lambda_{1}}{\lambda_{n}}.\] Indeed, the inequality holds when $\lambda_{i}$'s are close to each other. We can find element $a_{\lambda}\in A$ such that $a_{\lambda}$ is similar to the matrix \[\textrm{diag}(\underbrace{\lambda,\dots, \lambda}_{n},\underbrace{\lambda^{-1},\dots, \lambda^{-1}}_{n}).\]
As $\Acal=A\cap \Gamma_{0}$ is a lattice, using the similar argument with in the case of \textbf{SL}, we can find $\gamma\in A\cap\Gamma_{0}$ and $\lambda$ such that $a_{\lambda}$ is close with $\gamma$ (as elements in $A$) as much as we want. If $\lambda$ is sufficiently large and $\gamma\in A\cap \Gamma_{0}$ is sufficiently close to $a_{\lambda}$ then the $\gamma$ is a desired semisimple hyperbolic element.
Proposition \ref{prop:anos} says that $\alpha(\gamma)$ is an Anosov diffeomorphism. Furthermore, the above inequalities show that it satisfies conditions in the Brin--Manning theorem \ref{thm:BM}. Therefore, $M$ is homeomorphic to a torus or a flat manifold. Especially, there is a finite cover $M_{0}$ of $M$ so that $M_{0}$ is homeomorphic to a tori. This proves Theorem \ref{thm:main} in the cases of \textbf{Sp} and \textbf{SO}.

\section{Global rigidity}\label{sec:rig}
In this section, we will prove Corollary \ref{coro:main}. We will prove both topological rigidity and smooth rigidity. 

In the previous subsection, we proved that there is a finite cover $M_{0}$ of $M$, so that $M_{0}$ is homeomorphic to $d$-tori. Recall that we assume that there is a finite index subgroup $\Gamma_{0}<\Gamma$ such that we can lift $\Gamma_{0}$ action to $M_{0}$.

The $C^{0}$ conjugacy comes from theorems in \cite{BRHW17} for torus $M_{0}$ that is finitely covered $M$. Indeed, one also can use the result in \cite{MQ01} for the topological conjugacy. This is the only place we need lifting assumption to $M_{0}$. For $C^{\infty}$ conjugacy, we will use theorems in either \cite{FKS} or \cite{RHW}. As every hyperbolic element will give an Anosov diffeomorphism, the linear data has no rank one factor. Indeed, one can directly use the results in \cite{BRHW17} to prove $C^{\infty}$ regularity of the conjugacy.

\subsection{Topological conjugacy}

Recall that we denote the lifted $\Gamma_{0}$ action on $M_{0}$ as $\alpha_{1}$. Note that $\Gamma_{0}$ is still a lattice in the $G$. We can find hyperbolic element $\gamma\in\Gamma_{0}$. As we assumed fully supported invariant probability measure $\mu$ on $M$, there is a fully supported Borel probability measure $\mu_{0}$ on $M_{0}$ that is preserved by $\alpha_{1}(\Gamma_{0})$.  The \cite[Proposition 9.7]{BRHW17} implies that the action $\alpha_{1}$ lifts on the universal cover of $\widetilde{M}$.
Then, using \cite[Theorem 1.3]{BRHW17}, we can deduce that there is finite index subgroup $\Gamma_{1}<\Gamma_{0}$ so that $\alpha_{1}|_{\Gamma_{1}}$ is topologically conjugate to its linear data $\rho_{1}$ of $\alpha_{1}$. Indeed, the $\alpha_{1}(\gamma)$ is an Anosov diffeomorphism so that we have topological conjugacy $h$, i.e. \[\rho_{1}(\gamma)\circ h=h\circ \alpha_{1}(\gamma)\] for all $\gamma\in \Gamma_{1}$. This proves that there is a $C^{0}$ conjugacy between $\alpha_{1}$ and $\rho_{1}$ as in Corollary \ref{coro:main}.

\begin{rmk} After we know that the $M_{0}$ is homeomorphic to a torus and action lifts to the universal cover, one may be able to use \cite[Theorem 1.3]{MQ01} in order to get topological conjugacy.

\end{rmk}
\begin{rmk} In the case \textbf{SL}, we deduced that $M$ is itself torus. Therefore, using \cite[Proposition 9.7]{BRHW17}, we can lifts the action $\alpha$ into the universal cover of $M$ for some finite index subgroup since we assumed existence of invariant measure.
\end{rmk}

\subsection{Smooth conjugacy}

 We know that for every hyperbolic element in $\gamma\in \Gamma_{2}$, $\alpha_{1}(\gamma)$ is an Anosov diffeomorphism. Again, using \cite{PR01}, we can find free abelian subgroup $Z$ of rank $2$ in $\Gamma_{2}$ such that every element $z\in Z$ is hyperbolic. This implies that the linear data $\rho_{1}|_{Z}$ does not have rank one factor. (See \cite[Lemma 2.9]{RHW}.) So we can deduce that the $h$ is indeed smooth due to \cite{FKS} or \cite{RHW}. This proves smoothness of the conjugacy so that completes the proof of Corollary \ref{coro:main}. 
 \begin{rmk} One can directly use the results in \cite[Theorem 1.7]{BRHW17} to prove the smoothness of the conjugacy.
 \end{rmk}

\medskip
\bibliography{Global_rigidity_of_actions_by_higher_rank_lattices_with_dominated_splitting.bbl}{}
\bibliographystyle{alpha}

 \end{document}